\documentclass[11pt,oneside]{amsart}
\usepackage[english]{babel}
\usepackage{color}
\usepackage{fullpage}
\usepackage{amssymb}  % assumes amsmath package installed
\usepackage[utf8]{inputenc}

\usepackage[normalem]{ulem}
\usepackage{nicefrac}
\usepackage{hyperref}
\usepackage{enumerate}

\usepackage[runin]{abstract}

\usepackage{tikz}
\usetikzlibrary{patterns,calc,decorations.markings}
\colorlet{darkred}{red!85!black}
\colorlet{darkblue}{blue!70!black}

\usepackage[foot]{amsaddr}

\theoremstyle{plain}
\newtheorem{thm}{Theorem}[section]
\newtheorem{thm-int}{Theorem}
\newtheorem{lem}[thm]{Lemma}
\newtheorem{cor}[thm]{Corollary}

\newtheorem{prop}[thm]{Proposition}

\theoremstyle{definition}
\newtheorem{defn}[thm]{Definition}	
\newtheorem{example}[thm]{Example}
\newtheorem{rmk}[thm]{Remark}

\def\nat{\mathbb N}
\def\integ{\mathbb Z}
\def\real{\mathbb R}
\def\eps{\varepsilon}

\DeclareMathOperator\xLone{\text{L}^1}
\def\lcont{\xLone([0,T],\real^m)}
\def\lcontdue{\xLone([0,T],\real^{m+1})}
\DeclareMathOperator\xCinfty{\text{C}^\infty}
\DeclareMathOperator{\Span}{span}
\def\metr{\bf{g}}
\DeclareMathOperator{\ord}{ord}
\def\bigo{\mathcal O}

\def\drift{f_0}

\def\cost{\text{J}}
\DeclareMathOperator\val{V}
\DeclareMathOperator\tube{Tube}
\DeclareMathOperator\length{length}

\DeclareMathOperator\bsr{B_{SR}}
\DeclareMathOperator\dsrs{d_{SR}}
\newcommand{\bdt}[2]{\mathcal R^{\drift}_T(#1,#2)}

\DeclareMathOperator\distr{\Delta}
\newcommand{\bbox}[1]{\mathrm{Box}\left(#1\right)}

\DeclareMathOperator\maxtime{\mathcal T}
\DeclareMathOperator\maxtimectime{\delta_0}
\def\contset{{\mathcal U_{\maxtime}}}
\newcommand\contsett[1]{{\mathcal U_{#1}}}

\def\costuno{{\mathcal J}}
\def\costdue{{\mathcal I}}

\DeclareMathOperator\ccost{\Sigma_{int}}
\DeclareMathOperator\ccostcostuno{\Sigma_{int}^{\costuno}}
\DeclareMathOperator\ccostcostdue{\Sigma_{int}^{\costdue}}

\DeclareMathOperator\ctime{\sigma_{int}}
\DeclareMathOperator\ctimecostuno{\sigma_{int}^{\costuno}}
\DeclareMathOperator\ctimecostdue{\sigma_{int}^{\costdue}}
\DeclareMathOperator\ctimeaux{\omega}
\DeclareMathOperator\ctimeauxcostuno{\omega^{\costuno}}
\DeclareMathOperator\ctimeauxcostdue{\omega^{\costdue}}

\DeclareMathOperator\cneig{\sigma_{app}}
\DeclareMathOperator\cneigcostuno{\sigma_{app}^{\costuno}}
\DeclareMathOperator\cneigcostdue{\sigma_{app}^{\costdue}}

\DeclareMathOperator\capp{\Sigma_{app}}
\DeclareMathOperator\cappcostuno{\Sigma_{app}^{\costuno}}
\DeclareMathOperator\cappcostdue{\Sigma_{app}^{\costdue}}

\DeclareMathOperator\ltlctime{LTLC_{time}}
\DeclareMathOperator\ltlcneig{LTLC_{cost}}

\DeclareMathOperator\srsccost{\Sigma_{int}^{\text {SR-s}}}
\DeclareMathOperator\srsctime{\sigma_{int}^{\text {SR-s}}}

\DeclareMathOperator\srscapp{\Sigma_{app}^{\text {SR-s}}}

\DeclareMathOperator\srbccost{\Sigma_{int}^{\text {SR-b}}}
\DeclareMathOperator\srbctime{\sigma_{int}^{\text {SR-b}}}
\DeclareMathOperator\srbcneig{\sigma_{app}^{\text {SR-b}}}
\DeclareMathOperator\srbcapp{\Sigma_{app}^{\text {SR-b}}}

\def\crex{\overrightarrow{\text{exp}}\int}
\def\vett{\{f_1,\ldots,f_m\}}
\def\Lie{\text{Lie}}

\title[Complexity of control-affine motion planning]{Complexity of control-affine motion planning}
	\thanks{This work was  supported by the European Research Council, ERC StG 2009 ``GeCoMethods'', contract number 239748; by the ANR project {\it GCM}, program ``Blanche'', project number NT09\_504490; and by the Commission of the European Communities under the 7th Framework Programme Marie Curie Initial Training Network (FP7-PEOPLE-2010-ITN), project SADCO, contract number 264735}
	\author{F. Jean$^*$}% 
		\address{$^*$ ENSTA ParisTech, 828 Boulevard des Mar\'echaux
		91762 Palaiseau, France,  and Team  GECO, INRIA Saclay -- \^Ile-de-France,}
       \email{frederic.jean@ensta-paristech.fr}%
    \author{D. Prandi$^\dagger$$^\ddagger$}
		\address{$^\dagger$ Laboratoire LSIS, Université de Toulon, France, and Team GECO, INRIA-Centre de Recherche Saclay,}
  		\email{prandi@univ-tln.fr}%
    	\address{$^\ddagger$SISSA, Trieste, Italy}

\begin{document}

%%%%%%%%%%%%%%%%%%%%%%%%%%%%%%%%%%%%%%%%%%%%%%%%%%%%%%%%%%%%%%%%%%%%%%%%%%%%%%%%
\maketitle

\begin{abstract}
In this paper we study the complexity of the motion planning problem for control-affine systems. Such complexities are already defined and rather well-understood in the particular case of nonholonomic (or sub-Riemannian) systems. Our aim is to generalize these notions and results to systems with a drift. Accordingly, we present various definitions of complexity, as functions of the curve that is approximated, and of the precision of the approximation. Due to the lack of time-rescaling invariance of these systems, we consider geometric and parametrized curves separately. Then, we give some asymptotic estimates for these quantities.
\end{abstract}



% \tableofcontents

%%%%%%%%%%%%%%%%%%%%%%%%%%%%%%%%%%%%%%%%%%%%%%%%%%%%%%%%%%%%%%%%%%%%%%%%%%%%%%%%
\section{Introduction}

The concept of complexity was first developed for the non-holonomic motion planning problem in robotics.
Given a control system on a manifold $M$, the motion planning problem consists in finding an admissible trajectory connecting two points, usually under further requirements, such as obstacle avoidance.
If a cost function is given, it makes sense to try to find the trajectory costing the least.

Different approaches are possible to solve this problem (see \cite{laumond98}). 
Here we focus on those based on the following scheme:
\begin{enumerate}
	\label{scheme}
	\item find any (usually non-admissible) curve or path solving the problem,
	\item approximate it with admissible trajectories. 
\end{enumerate}
The first step is independent of the control system, since it depends only on the topology of the manifold and of the obstacles, and it is already well understood (see \cite{Schwartz1983}).
Here, we are interested in the second step, which depends only on the local nature of the control system near the path.
The goal of the paper is to understand how to measure the complexity of the approximation task.
By complexity we mean a function of the non-admissible curve $\Gamma\subset M$ (or path $\gamma:[0,T]\to M$), and of the precision of the approximation, quantifying the difficulty of the latter by means of the cost function.

\subsection{Control theoretical setting} % (fold)
\label{sub:control_theoretical_setting}

% A sub-Riemannian (or non-holonomic) control system on a smooth manifold $M$ is a control system in the form
% 	\begin{equation}\label{int:sr}%\tag{SR}
% 		\dot q(t) = \sum_{i=1}^m u_i(t) \,f_i(q(t)),\qquad \text{a.e. }t\in[0,T],
% 	\end{equation}
% where $u:[0,T]\to \real^m$ is an integrable and bounded control function and $\vett$ is a family of smooth vector fields on $M$ satisfying the H\"ormander condition, i.e. such that its iterated Lie brackets generate the whole tangent space at every point.
% Moreover, we will always assume the sub-Riemannian structure to be equiregular (see Section~\ref{sec:sr}).
% Given a sub-Riemannian control system, a natural choice for the cost is the $L^1$-norm of the controls. 
% Due to the linearity and the reversibility in time of such a system, the associated value function is in fact a distance, called Carnot-Carath\'eodory distance, that endows $M$ with a metric space structure.

In this paper, we consider a control-affine system on a smooth manifold $M$ is a control system in the form
\begin{equation}\label{int:d}
	\dot q(t) = f_0(q(t))+ \sum_{i=1}^m u_i(t) \,f_i(q(t)),\qquad \text{a.e. }t\in[0,T],
\end{equation}
where $u:[0,T]\rightarrow\real^m$ is an integrable control function and $\drift,f_1,\ldots, f_m$ are (not necessarily linearly independent) smooth vector fields.
The uncontrolled vector field $\drift$ is called the \em drift\em.
These kind of systems appear in plenty of applications. 
As examples we cite: mechanical systems with controls on the acceleration (see e.g., \cite{Bullo2004,Sigalotti2010}), where the drift is the velocity, quantum control (see e.g., \cite{D'Alessandro2008,Boscain2006}), where the drift is the free Hamiltonian, or the swimming of microscopic organisms (see, e.g., \cite{Tucsnak2007}).
We always assume the strong H\"ormander condition, i.e., that the iterated Lie brackets of the controlled vector fields $f_1,\ldots,f_m$ generate the whole tangent space at any point.
This guarantees the small time local controllability of system \eqref{int:d} and allows us to associate to \eqref{int:d} a sub-Riemannian control system.
Such assumption is generically satisfied, e.g., by finite-dimensional quantum control systems with two controls, as the ones studied in \cite{Boscain2004,Boscain2002,D'alessandro2001}.
Except when explicitly stated, we do not make any assumption on the dimension of $\Span\{f_1(q),\ldots,f_m(q)\}$ which, in particular, can depend on the point $q\in M$.

When posing $\drift=0$ in \eqref{int:d} we obtain the (small) sub-Riemannian control system associated with \eqref{int:d}, i.e., the driftless control system in the form
\begin{equation}\label{int:sr}
	\dot q(t) = \sum_{i=1}^m u_i(t) \,f_i(q(t)),\qquad \text{a.e. }t\in[0,T],
\end{equation}
This system satisfies the H\"ormander condition is satisfied, i.e., the iterated Lie brackets of the vector fields $f_1,\ldots,f_m$ generate the whole tangent space at any point.
Moreover, we will always assume the sub-Riemannian structure to be equiregular (see Section~\ref{sec:sr}).
Given a sub-Riemannian control system, a natural choice for the cost is the $\xLone$-norm of the controls. 
Due to the linearity and the reversibility in time of such a system, the associated value function is in fact a distance, called Carnot-Carath\'eodory distance, that endows $M$ with a metric space structure.

% In this paper, we focus on a very important generalization of control system \eqref{int:sr}, namely on control-affine systems. These systems are obtained by adding to \eqref{int:sr} an uncontrolled vector field $\drift$, called the \em drift\em,  and are in the form
% \begin{equation}\label{int:d}
% 	\dot q(t) = f_0(q(t))+ \sum_{i=1}^m u_i(t) \,f_i(q(t)),\qquad \text{a.e. }t\in[0,T].
% \end{equation}
% These kind of systems appears in plenty of applications. As an example we cite, mechanical systems with controls on the acceleration (see e.g., \cite{Bullo2004}, \cite{Sigalotti2010}), where the drift is the velocity, or quantum control (see e.g., \cite{D'Alessandro2008}, \cite{Boscain2006}), where the drift is the free Hamiltonian.
% We always assume the strong H\"ormander condition, i.e., that the family $\vett$ satisfies the H\"ormander condition.

Our work will focus on the following cost functions, 
\begin{equation}
	\label{int:costfunctions}
	\costuno(u,T) = \int_0^T \sqrt{\sum_{j=1}^m u_j(t)^2}\,dt \quad\text{ and }\quad
	\costdue (u,T)= \int_0^T \sqrt{1+\sum_{j=1}^m u_j(t)^2}\,dt.
\end{equation}
Namely, $\costuno(u,T)=\|u\|_{\lcont}$ and $\costdue(u,T)=\|(1,u)\|_{\lcontdue}$.
Let $q\in M$ and define $q_u:[0,T]\to M$ as the trajectory associated with a control $u\in\lcont$ such that $q_u(0)=q$.
The cost $\costuno$, measuring the $\xLone$-norm of the control, quantifies the cost spent by the controller to steer the system \eqref{int:d} along $q_u$.
On the other hand, $\costdue$ measures the Riemannian length of $q_u$ with respect to a Riemannian metric\footnote{Whenever it exists, which is not always the case since the vector fields $\drift,f_1,\ldots,f_m$ could be not linearly independent.} such that $\drift,f_1,\ldots,f_m$ are orthonormal.

Fix a time $\maxtime>0$ and consider the two value functions $\val^\costuno(q,q')$ and $\val^\costdue(q,q')$ as the infima of the costs $\costuno$ and $\costdue$, respectively, over all controls $u\in\contset=\bigcup_{0<T\le\maxtime} \xLone([0,T],\real^m)$ steering the system from $q$ to $q'$.
Contrarily to what happens in sub-Riemmanian geometry with the Carnot-Carath\'eodory distance, these value functions are not symmetric, and hence do not induce a metric space structure on $M$.
In fact, system \eqref{int:d} is not reversible -- i.e., changing orientation to an admissible trajectory does not yield an admissible trajectory.

We consider controls defined on $T\le\maxtime$ since we are interested in the local behavior of system \eqref{int:d}.
Indeed, without an upper-bound for the time of definition of the controls, the reacheable sets $\mathcal R^{\drift}(q,\eps)=\{q'\in 
M \mid \val^\costuno(q,q')\le\eps \}$ are in general non-compact for any $q\in M$ and $\eps>0$.
As a byproduct of this choice, by taking $\maxtime$ sufficiently small, it is then possible to prevent any exploitation of the geometry of the orbits of the drift (that could be, for example, closed).
Let us also remark that, since the controls can be defined on arbitrarily small times, it is possible to approximate admissible trajectories for system \eqref{int:d} via trajectories for the sub-Riemannian associated system (i.e., the one obtained by posing $\drift\equiv0$) rescaled on small intervals.

% subsection control_theoretical_setting (end)

\subsection{Complexities} % (fold)
\label{sub:complexities}

Heuristically, the complexity of a curve $\Gamma$ (or path $\gamma:[0,T]\to M$) at precision $\eps$ is defined as the ratio
\begin{equation}
	\label{eq:euristic}
	\frac {\text{``cost'' to track $\Gamma$ at precision $\eps$}} {\text{``cost'' of an elementary $\eps$-piece}}.
\end{equation}
In order to obtain a precise definition of complexity, we need to give a meaning to the notions appearing above.
Namely, we have to specify what do we mean by ``cost''\footnote{The cost appearing in \eqref{eq:euristic} is not necessarily related with the cost function ($\costuno$ or $\costdue$) taken into account. This is the reason for the quotation marks.}, tracking at precision $\eps$, and elementary $\eps$-piece.
Indeed, these choices will depend on the type of motion planning problem at hand.

First of all, we classify motion planning problems as \emph{time-critical} or \emph{static}, depending on wether the constraints depend on time or not.
The typical example of static motion planning problem is the obstacle avoidance problem with fixed obstacles.
On the other hand, the same problem where the position of the obstacles depends on time, or the rendez-vous problems, are examples of time-critical motion planning problems.
% Clearly, the notion of complexities best suited for these types of problems are different.

For static motion planning problems, the solution of the first step of the motion planning scheme (introduced at the beginning of the paper) is usually given as a curve, i.e., a dimension $1$ connected submanifolds of $\Gamma\subset M$ diffeomorphic to a closed interval.
On the other hand, in time-critical problems we have to keep track of the time.
Thus, for this type of problems, the solution of the first step is a path, i.e., a smooth injective function $\gamma:[0,T]\to M$.
As a consequence, when computing the complexity of paths we will require the approximating trajectories to respect also the parametrization, and not only the geometry, of the path.
While in the sub-Riemannian case, due to the time rescaling properties of the control system, these concepts coincide, this is not the case for control-affine systems.

In this work, we consider four distinct notions of complexity, two for curves (static problems) and two for paths (time-critical problems). 
In both cases, one of the two will be based on the interpolation of the given curve or path, while the other will consider trajectories that stays near the curve or path.
Thus, for this complexity, we will need to fix a metric. 
In this work we will consider only the sub-Riemannian metric of the associated sub-Riemannian control system (\ref{int:sr})).

We remark that the two complexity for curves are the same as the sub-Riemannian ones already introduced in \cite{Gromov1996,Jean2001}.
This is true also for what we call the neighboring approximation complexity of a path, since in the sub-Riemannian case it coincides with the tubular approximation complexity.
On the other hand, what we call the interpolation by time complexity never appeared in the literature, to our knowledge.
Here we give the definitions for a generic cost $J:\contset\to [0,+\infty)$.

Fix a curve $\Gamma$ and,
%-- i.e., a one dimensional oriented submanifold with boundary of $M$ diffeomorphic to a closed interval.
for any $\eps>0$, define the following complexities for $\Gamma$. 
\begin{itemize}
	\item \emph{Interpolation by cost complexity:} (see Figure~\ref{fig:ccost})
		\begin{figure}[tb]
			\begin{center}
				\begin{minipage}[t]{0.45\textwidth}
					\centering
					%!tikz editor 1.0

%!tikz preamble begin
\usetikzlibrary{calc}
%!tikz preamble end

%!tikz source begin
\begin{tikzpicture}
[rotate=-30,scale=.75]
				      %\draw (-1,-1) grid (7,7);
				      %\draw[step=.5,gray,very thin] (-1,-1) grid (7,7);
				      \draw (0,0) node[below] {$x$}
				        to[in=245,out=45] node[midway,above left] {$\Gamma$} (.5,1.5)
				        to[out=65,in=200] (1,2.5)
				        to[in=240,out=20]  (3.5,3)
				        to[out=60,in=250] (4.5,4)
				        to[out=70,in=215] (5,5.5)
				        to[out=35,in=190]  (6,6)
				        node[right] {$y$};

					%punti di interpolazione
					\node (qu1) at (.33,.82) {};
						% [label=right:$q_u(t_1)$] {};
				    \filldraw[red] (qu1) circle (1pt);
					\node (qu2) at (1,2.5) {};
					\node[above right] at (qu2) {$q_u(t_{i-1})$};
					\filldraw[red] (qu2) circle (1pt);
					\node (qu3) at (3.5,3)  {};
					\filldraw[red] (qu3) circle (1pt);
					\node[below right,yshift=2pt] at (qu3) {$q_u(t_i)$};
					\node (qu4) at (4.41,3.83)  {};
					% \node[above=3.5pt,xshift=-7pt] at (qu4) {$q_u(t_4)$};
					\filldraw[red] (qu4) circle (1pt);
					\node (qu5) at (4.6,4.36) {};
						% [label=right:$q_u(t_5)$] {};
					\filldraw[red] (qu5) circle (1pt);
					\node (qu6) at (5.15,5.6)  {};
					% \node[above right] at (qu6) {$q_u(t_6)$};
					\filldraw[red] (qu6) circle (1pt);
					
					%eps
					\draw (qu2) -- ($(qu2)+2.6*(.3,-.5)$);
					\draw ($(qu2)+2.6*(.3,-.5)$) -- node[midway,below] {$\le\varepsilon$}
					($(qu3)+1.2*(.3,-.5)$);
					\draw (qu3) -- ($(qu3)+1.2*(.3,-.5)$);

				    %curva interpolante
				      \draw[red] plot[smooth] coordinates {(0,0) (.5,.5) (.3,1) (.5,2.5) (1,2.5) (1.5,2.3) (2,2.5) (3.5,2.5) (3.5,3.5) (4.7,4) (4.3,5) (5,5.7) (5.4,5.3) (6,6)};
				      % \node at (0,3.5) {$q_u(\cdot)$};
				      % \draw[->] (.2,3.2) -- (.5,2.7);
\end{tikzpicture}
%!tikz source end
				\caption{Interpolation by cost complexity}
				\label{fig:ccost}
				\end{minipage}	
				\hfill
				\begin{minipage}[t]{0.45\textwidth}
					\centering
				    %!tikz editor 1.0
%!tikz source begin
\begin{tikzpicture}
[rotate=-40,scale=.75]
				      %\draw (-1,-1) grid (7,7);
				      %\draw[step=.5,gray,very thin] (-1,-1) grid (7,7);
				      \draw (0,0) node[below] {$x$}
				        to[in=245,out=45] (.5,1.5)
				        to[out=65,in=200] (1,2.5)
				        to[in=240,out=20] (3.5,3)
				        to[out=60,in=250] (4.5,4)
				        to[out=70,in=215] (5,5.5)
				        to[out=35,in=190] node[above] {$\Gamma$} (6,6)
				        node[right] {$y$};

				      % tubo
				      \foreach \x in {1,-1}
				      {
				      \draw[dashed] ($(0,0)+(\x*.5,-\x*.5)$)
				        to[in=245,out=45] ($(.5,1.5)+(\x*.5,-\x*.5)$)
				        to[out=65,in=200] ($(1,2.5)+(\x*.5,-\x*.5)$)
				        to[in=240,out=20] ($(3.5,3)+(\x*.5,-\x*.5)$)
				        to[out=60,in=250] ($(4.5,4)+(\x*.5,-\x*.5)$)
				        to[out=70,in=215] ($(5,5.5)+(\x*.5,-\x*.5)$)
				        to[out=35,in=190] ($(6,6)+(\x*.5,-\x*.5)$) ;
				      }
				      \draw[dashed] (-.5,.5) arc (125:317:.71);
				      \draw[dashed] (6.5,5.5) .. controls (7.3,6) and (6.5,6.7) .. (5.5,6.5);
				      \node[below] at (2.5,1) {Tube$(\Gamma,\varepsilon)$};
				      \draw[->] (2.5,1) -- (3,1.95);

				      %curva interpolante
				      \draw[red] plot[smooth] coordinates {(0,0) (.5,.5) (.3,1) (.5,2.5) (1,2.5) (1.5,2.3) (2,2.5) (3.5,2.5) (3.5,3.5) (4.7,4) (4.3,5) (5,5.7) (5.4,5.3) (6,6)};
				      \node at (0,3.5) {$q_u(\cdot)$};
				      \draw[->] (.2,3.2) -- (.5,2.7);
\end{tikzpicture}
%!tikz source end
				\caption{Tubular approximation complexity}
				\label{fig:capp}
				\end{minipage}

				\vspace{1cm}

				\begin{minipage}{0.45\textwidth}
					\centering
					%!tikz editor 1.0

%!tikz preamble begin
\usetikzlibrary{calc}
%!tikz preamble end

%!tikz source begin
\begin{tikzpicture}
[rotate=-30,scale=.75]
				      %\draw (-1,-1) grid (7,7);
				      %\draw[step=.5,gray,very thin] (-1,-1) grid (7,7);
				      \draw (0,0) node[below] {$x$}
				        to[in=245,out=45] %node[midway,above left] {$\gamma(\cdot)$} 
				        (.5,1.5)
				        to[out=65,in=200] (1,2.5)
				        to[in=240,out=20]  (3.5,3)
				        to[out=60,in=250] (4.5,4)
				        to[out=70,in=215] (5,5.5)
				        to[out=35,in=190]  (6,6)
				        node[right] {$y$};

					%punti di interpolazione
					\node (qu1) at (.33,.82) {};
						% [label=right:$q_u(t_1)$] {};
				    \filldraw[red] (qu1) circle (1pt);
					\node (qu2) at (1,2.5) {};
					\node[above] at ($(qu2)+(0,1)$) {$q_u(t_{i-1})=\gamma(t_{i-1})$};
					\filldraw[red] (qu2) circle (1pt);
					\draw[->] ($(qu2)+(0,1)$) -- (qu2);
					\node (qu3) at (3.5,3)  {};
					\filldraw[red] (qu3) circle (1pt);
					\node[below] at ($(qu3)+(2,.7)$) {$q_u(t_i)=\gamma(t_i)$};
					\draw[->] ($(qu3)+(1,0)$)  -- (qu3);
					\node (qu4) at (4.41,3.83)  {};
					% \node[above=3.5pt,xshift=-7pt] at (qu4) {$q_u(t_4)$};
					\filldraw[red] (qu4) circle (1pt);
					\node (qu5) at (4.6,4.36) {};
						% [label=right:$q_u(t_5)$] {};
					\filldraw[red] (qu5) circle (1pt);
					\node (qu6) at (5.15,5.6)  {};
					% \node[above right] at (qu6) {$q_u(t_6)$};
					\filldraw[red] (qu6) circle (1pt);
					
					%eps
					\draw (qu2) -- ($(qu2)+2.6*(.3,-.5)$);
					\draw ($(qu2)+2.6*(.3,-.5)$) -- node[midway,below] {on avg. $\le\varepsilon$}
					($(qu3)+1.2*(.3,-.5)$);
					\draw (qu3) -- ($(qu3)+1.2*(.3,-.5)$);

				    %curva interpolante
				      \draw[red] plot[smooth] coordinates {(0,0) (.5,.5) (.3,1) (.5,2.5) (1,2.5) (1.5,2.3) (2,2.5) (3.5,2.5) (3.5,3.5) (4.7,4) (4.3,5) (5,5.7) (5.4,5.3) (6,6)};
				      % \node at (0,3.5) {$q_u(\cdot)$};
				      % \draw[->] (.2,3.2) -- (.5,2.7);
\end{tikzpicture}
%!tikz source end
					\caption{Time interpolation complexity}
					\label{fig:ctime}
				\end{minipage}
                \hfill
				\begin{minipage}{0.45\textwidth}
					\centering
					%!tikz editor 1.0
%!tikz source begin
\begin{tikzpicture}[rotate=-30,scale=.75]
				      %\draw (-1,-1) grid (7,7);
				      %\draw[step=.5,gray,very thin] (-1,-1) grid (7,7);
				      \draw (0,0) node[below] {$x$}
				        to[in=245,out=45] (.5,1.5)
				        to[out=65,in=200] (1,2.5)
				        to[in=240,out=20] (3.5,3)
				        to[out=60,in=250]  (4.5,4)
				        to[out=70,in=215] (5,5.5)
				        to[out=35,in=190]  (6,6) node[right] {$y$};
				      % palla
				      \node (v1) at (4.5,4) {};
				      \draw[scale=2.5,dashed] ($(v1)+(-0.5,0)$)
				            to[out=90,in=180] ($(v1) + (0,0.5)$)
				            to[out=0,in=90] ($(v1) + (.5,0)$)
				            to[out=270,in=0] ($(v1) + (.2,-0.4)$)
				            to[out=180,in=270] ($(v1)+(-0.5,0)$);
				      %curva interpolante
				      \draw[red] plot[smooth] coordinates {(0,0) (.5,.5) (.3,1) (.5,2.5) (1,2.5) (1.5,2.3) (2,2.5) (3.5,2.5) (3.5,3.5) (4.2,4.3) };
                      \draw[red,dashed] plot[smooth] coordinates {(4.2,4.3) (5,5) (5.2,5.6) (5.6,6.05) (6,6)};
				      %\node at (0,3.5) {$q_u(\cdot)$};
				      %\draw[->] (.2,3.2) -- (.5,2.7);
				      \node (qu) at (4.2,4.3) {};
				      \filldraw[red] (qu) circle (1pt);
				      \filldraw (v1) circle (1pt);
				      \node[above] (qu) at (4.2,4.3) {$q_u(t)$};
				      \node[right] at ($(v1)+(0,-.2)$) {$\gamma(t)$};
				      \node at (4.6,6) {$\bsr(\gamma(t),\varepsilon)$};
				      %\draw[->] (2.5,5) to (3.5,4.5);
\end{tikzpicture}
%!tikz source end
                    \caption{Neighboring approximation complexity}
					\label{fig:cneig}
				\end{minipage}
		  	\end{center}
		\end{figure}
		For $\eps>0$, let an \em $\eps$-cost interpolation \em\/ of $\Gamma$ to be any control $u\in\contset$ such that there exist $0=t_0<t_1<\ldots<t_N=T\le\maxtime$ for which the trajectory $q_u$ with initial condition $q_u(0)=x$ satisfies $q_u(T)= y$, $q_u(t_i)\in\Gamma$ and $J(u|_{[t_{i-1},t_{i})},t_{i}-t_{i-1})\le\eps$, for any $i=1,\ldots, N$.
		Then, let
		\[
			\ccost(\Gamma,\eps) = \frac 1 \eps\inf \big\{ J(u,T)\mid u\text{ is an $\eps$-cost interpolation of }\Gamma \big\}.
		\]
		This function measures the number of pieces of cost $\eps$ necessary to interpolate $\Gamma$.
		Namely, following a trajectory given by a control admissible for $\ccost(\Gamma,\eps)$, at any given moment it is possible to go back to $\Gamma$ with a cost less than $\eps$. 
	\item \emph{Tubular approximation complexity:} (see Figure~\ref{fig:capp})
		Let $\tube(\Gamma,\eps)$ to be the tubular neighborhood of radius $\eps$ around the curve $\Gamma$ w.r.t.\ the small sub-Riemannian system associated with \eqref{int:d} (obtained by putting $\drift\equiv0$, see Section~\ref{sub:control-affine_systems}), and define
		\[
			\capp(\Gamma,\eps) = \frac 1 \eps \inf \left\{ J(u,T) \left|\: 
			\begin{array}{l}
				0<T\le\maxtime,\\
				q_u(0)=x,\, q_u(T)=y, \\
				q_u\big([0,T] \big)\subset \tube(\Gamma,\eps) \\
			\end{array}
			\right.
		\right\}
		\]
		This complexity measures the number of pieces of cost $\eps$ necessary to go from $x$ to $y$ staying inside the sub-Riemannian tube $\tube(\Gamma,\eps)$.
		Such property is especially useful for motion planning with obstacle avoidance.
		In fact, if the sub-Riemannian distance of $\Gamma$ from the obstacles is at least $\eps_0>0$, then trajectories obtained from controls admissible for $\capp(\Gamma,\eps)$, $\eps<\eps_0$, will avoid such obstacles.
\end{itemize}

We then define the following complexities for a path $\gamma:[0,T]\to M$ at precision $\eps>0$.
\begin{itemize}
	\item \emph{Interpolation by time complexity:} (see Figure~\ref{fig:ctime})
		Let a $\delta$-time interpolation of $\gamma$ to be any control $u\in\lcont$ such that its trajectory $q_u:[0,T]\to M$ in \eqref{int:d} with $q_u(0)=\gamma(0)$ is such that $q_u(T)=\gamma(T)$ and that, for any interval $[t_0,t_1]\subset[0,T]$ of length $t_1-t_0\le \delta$, there exists $t\in[t_0,t_1]$ with $q_u(t)=\gamma(t)$.
		Then, fix a $\maxtimectime>0$ and let
		\[
			\ctime(\gamma,\eps) = \inf \left\{ \frac T \delta \, \bigg|\: 
				\begin{array}{c}
					\delta\in(0,\maxtimectime) \text{ and exists $u\in\lcont$,} \\
					\delta\text{-time interpolation of $\gamma$, s.t. } \delta\,J(u,T)\le\eps
				\end{array}
			\right\}.
		\]
		Controls admissible for this complexity will define trajectories such that the minimal average cost between any two consecutive times such that $\gamma(t)=q_u(t)$ is less than $\eps$.
		It is thus well suited for time-critical applications where one is interested in minimizing the time between the interpolation points - e.g.\ motion planning in rendez-vous problem.
	\item \emph{Neighboring approximation complexity:} (see Figure~\ref{fig:cneig})
		Let $\bsr(p,\eps)$ denote the ball of radius $\eps$ centered at $p\in M$ w.r.t.\ the small sub-Riemannian system associated with \eqref{int:d} (obtained by putting $\drift\equiv0$, see Section~\ref{sub:control-affine_systems}), and define
		\[
			\cneig(\gamma,\eps) = \frac 1 \eps \inf \left\{ J(u,T) \bigg|\: 
			\begin{array}{l}
				q_u(0)=x,\, q_u(T)= y, \\
				q_u(t) \in \bsr(\gamma(t),\eps),\,\forall t\in[0,T]\\
			\end{array}
			\right\}.
		\]
		This complexity measures the number of pieces of cost $\eps$ necessary to go from $x$ to $y$ following a trajectory that at each instant $t\in[0,T]$ remains inside the sub-Riemannian ball $\bsr(\gamma(t),\eps)$.
		Such complexity can be applied to motion planning in rendez-vous problems where it is sufficient to attain the rendez-vous only approximately, or to motion planning with obstable avoidance where the obstacles are moving.
\end{itemize}

Whenever we will need to specify with respect to which cost a complexity is measured, we will write the cost function as apex -- e.g., to we will denote the interpolation by cost complexity w.r.t.\ $\costuno$ as $\ccost^\costuno$.

We remark that for the interpolation by time complexity the ``cost'' in \eqref{eq:euristic} is the time, while for all the other complexities it is the cost function associated with the system.
For the motivation of the bound on $\delta$ in the definition of the interpolation by time complexity, see Remark~\ref{rmk:timeboundmot}.
Finally, whenever a metric is required, we use the sub-Riemannian one.
Although such metric is natural for control-affine systems satisfying the H\"ormander condition, nothing prevents from defining complexities based on different metrics.

Two functions $f(\eps)$ and $g(\eps)$, tending to $\infty$ when $\eps\downarrow0$ are \emph{weakly equivalent} (denoted by $f(\eps)\asymp g(\eps)$) if both $f(\eps)/g(\eps)$ and $g(\eps)/f(\eps)$, are bounded when $\eps\downarrow0$.
When $f(\eps)/g(\eps)$ (resp. $g(\eps)/f(\eps)$) is bounded, we will write $f(\eps)\preccurlyeq g(\eps)$ (resp. $f(\eps)\succcurlyeq g(\eps)$).
In the sub-Riemannian context, the complexities are always measured with respect to the $\xLone$cost of the control, $\costuno$.
Then, for any curve $\Gamma\subset M$ and path $\gamma:[0,T]\to M$ such that $\gamma([0,T])=\Gamma$ it holds $\ccostcostuno(\Gamma,\eps) \asymp \cappcostuno(\Gamma,\eps) \asymp \cneigcostuno(\gamma,\eps)$.

A complete characterization of weak asymptotic equivalence of sub-Riemannian complexities is obtained in \cite{Jean2003a}.
We state here such characterization in the special case where $\vett$ defines an equiregular structure.

\begin{thm}
	\label{thm:srcomplexity}
    Assume that $\vett$ defines an equiregular sub-Riemannian structure. 
    Let $\Gamma\subset M$ be a curve and $\gamma:[0,T]\to M$ be a path such that $\gamma([0,T])=\Gamma$.
    Then, if there exists $k\in\nat$ such that $T_q\Gamma\subset \distr^k(q)\setminus\distr^{k-1}(q)$ for any $q\in \Gamma$, it holds
    \[
    	\ccost(\Gamma,\eps)\asymp \capp(\Gamma,\eps) \asymp \cneig(\gamma,\eps)\asymp \frac 1 {\eps^k}.
    \] 
    Here the complexities are measured w.r.t.\ the cost $\costuno(u,T)=\|u\|_{\lcont}$.
\end{thm}
We mention also that for a restricted set of sub-Riemannian systems, i.e., one-step bracket generating or with two controls and dimension not larger than $ 6$, strong asymptotic estimates and explicit asymptotic optimal syntheses are obtained in the series of papers \cite{8,9,10,11,12,13,14} (see \cite{gauthierbonnard} for a review).

\subsection{Main results} % (fold)
\label{sub:main_results}

Our first result is the following.
It completes the description of the sub-Riemannian weak asymptotic estimates started in Theorem~\ref{thm:srcomplexity}, describing the case of the interpolation by time complexity.
It is proved in Section~\ref{sec:first}.

\begin{thm}
	\label{thm:srtimecomplexity}
    Assume that $\vett$ defines an equiregular sub-Riemannian structure and let $\gamma:[0,T]\to M$ be a path.
    Then, if there exists $k\in\nat$ such that $\dot\gamma(t)\in\distr^k(\gamma(t))\setminus\distr^{k-1}(\gamma(t))$ for any $t\in[0,T]$, it holds
    \[
    	\ctime(\gamma,\eps)\asymp \frac 1 {\eps^{k}}.
    \] 
    Here the complexity is measured w.r.t.\ the cost $\costuno(u,T)=\|u\|_{\lcont}$.
\end{thm}

Since in the sub-Riemannian context one is only interested in the cost $\costuno$, Theorems~\ref{thm:srcomplexity} and \ref{thm:srtimecomplexity} completely characterize the weak asymptotic equivalences of complexities of equiregular sub-Riemannian manifolds.

The main result of the paper is then a weak asymptotic equivalence of the above defined complexities in control-affine systems, generalizing Theorems~\ref{thm:srcomplexity} and \ref{thm:srtimecomplexity}.

\begin{thm}
	\label{thm:driftcomplexity}
	Assume that $\vett$ defines an equiregular sub-Riemannian structure and that 
    $\drift\subset\distr^s\setminus\distr^{s-1}$ for some $s\ge 2$.
    Also, assume that the complexities are measured w.r.t.\ the cost function $\costuno(u,T)=\|u\|_{\lcont}$ or $\costdue(u,T)=\|(1,u)\|_{\lcontdue}$.
    We then have the following.
    \begin{enumerate}[i.]
    	\item Let $\Gamma\subset M$ be a curve and define
		    $\kappa=\max\{ k \colon\: T_p\Gamma\in\Delta^k(p)\setminus\Delta^{k-1}(p),\, $ for any $p$ in an open subset of $\Gamma\}$.
		    Then, whenever the maximal time of definition of the controls $\maxtime$ is sufficiently small, it holds
		    \[
		    	\ccost(\Gamma,\eps) \asymp \capp(\Gamma,\eps)\asymp \frac 1 {\eps^\kappa}.
		    \]
    	\item On the other hand, let $\gamma:[0,T]\to M$ be a path such that $\drift(\gamma(t))\neq \dot\gamma(t) \mod\Delta^{s-1}(\gamma(t))$ for any $t\in [0,T]$ and define
    		$\kappa=\max\{ k \colon\: \gamma(t)\in\Delta^k(\gamma(t))\setminus\Delta^{k-1}(\gamma(t)) $ for any $t$ in an open subset of $[0,T]\}$. 
		    Then, it holds
			\[
				\ctime(\gamma,\eps)\asymp\cneig(\gamma,\eps) \asymp \frac 1 {\eps^{\max\{\kappa,s\}}},
			\]
			where the first asymptotic equivalence is true only when $\maxtimectime$, i.e., the maximal time-step in $\ctime(\gamma,\eps)$, is sufficiently small.
    \end{enumerate}
    \end{thm}

This theorem shows that, asymptotically, the complexity of curves is not influenced by the drift, and only depends on the underlying sub-Riemannian system, while the one of paths depends also on how ``bad'' the drift is with respect to this system.
We remark also that for the neighboring approximation complexity, $\cneig$, it is not necessary to have an a priori bound on $\maxtime$.

% subsection complexities (end)

\subsection{Long time local controllability} % (fold)
\label{sub:long_time_local_controllability}

As an application of the above theorem, let us briefly mention the problem of \emph{long time local controllability} (henceforth simply \emph{LTLC}), i.e., the problem of staying near some point for a long period of time $T>0$.
This is essentially a stabilization problem around a non-equilibrium point.

Since the system \eqref{int:d} satisfies the strong H\"ormander condition, it is always possible to satisfy some form of LTLC\@. 
Hence, it makes sense to quantify the minimal cost needed, by posing the following. 
% (To lighten the notation, we consider only the cost $\costuno$.)
Let $T>0$, $q_0\in M$, and $\gamma_{q_0}:[0,T]\to M$, $\gamma_{q_0}(\cdot)\equiv q_0$. 
\begin{itemize}
	\item \emph{LTLC complexity by time:}
		\begin{equation*}
			\ltlctime(q_0,T,\eps)=\ctime(\gamma_{q_0},\eps).
		\end{equation*}
		Here, we require trajectories defined by admissible controls to pass through $q_0$ at intervals of time such that the minimal average cost between each passage is less than $\eps$.
	\item \emph{LTLC complexity by cost:}
		\begin{equation*}
			\ltlcneig(q_0,T,\eps)=\cneig(\gamma_{q_0},\eps).
		\end{equation*}
		Admissible controls for this complexity, will always be contained in the sub-Riemannian ball of radius $\eps$ centered at $q_0$.
\end{itemize}

Clearly, if $\drift(q_0)=0$, then $\ltlctime(q_0,T,\delta)=\ltlcneig(q_0,T,\eps)=0$, for any $\eps,\delta,T>0$. 
Although $\gamma_{q_0}$ is not a path by our definition, since it is not injective and $\dot\gamma_{q_0}\equiv0$, the arguments of Theorem~\ref{thm:driftcomplexity} can be applied also to this case.
Hence, we get the following asymptotic estimate for the LTLC complexities. 

\begin{cor}
	\label{thm:longcontr}
	Assume that $\vett$ defines an equiregular sub-Riemannian structure and that 
    $\drift\subset\distr^s\setminus\distr^{s-1}$ for some $s\ge 2$.
    Also, assume that the complexities are measured w.r.t.\ the cost function $\costuno(u,T)=\|u\|_{\lcont}$ or $\costdue(u,T)=\|(1,u)\|_{\lcontdue}$.
	Then, for any $q_0\in M$ and $T>0$ it holds 
	\[
		\ltlctime(q_0,T,\delta)\asymp  \ltlcneig(q_0,T,\eps)\asymp \frac 1 {\eps^s}.
	\]
\end{cor}

% subsection long_time_local_controllability (end)

\subsection{Structure of the paper} % (fold)
\label{sub:structure_of_the_paper}

In Section~\ref{sec:prelim} we introduce more in detail the setting of the problem. 
In Section~\ref{sec:normal_forms} we present some technical results regarding families of coordinates depending continuously on the base point.
These results will be essential in the sequel.
Section~\ref{sec:cost} collects some useful properties of the costs $\costuno$ and $\costdue$, proved mainly in \cite{prandi2013}, while Section~\ref{sec:first} is devoted to relate the complexities of the control-affine system with those of the associated sub-Riemannian systems, and to prove Theorem~\ref{thm:srtimecomplexity}.
In this section we also prove Proposition~\ref{prop:firstestimate}, that gives a first result in the direction of Theorem~\ref{thm:driftcomplexity} showing when the sub-Riemannian and control-affine complexities coincide.
Finally, the proof of the main result is contained in Sections~\ref{sec:complexity_for_curves_under_h_s_} and \ref{sec:complexity_of_paths_under_h_s_}, for curves and paths respectively.

% subsection structure_of_the_paper (end)

\section{Preliminaries}
\label{sec:prelim}

Throughout this paper, $M$ is an $n$-dimensional connected smooth manifold. 

\subsection{Sub-Riemannian control systems}
	\label{sec:sr}
	As already stated, a sub-Riemannian (or non-holonomic) control system on a connected smooth manifold $M$ is a control system in the form 
	\begin{equation}\label{cs:sr}\tag{SR} 
		\dot q(t) = \sum_{i=1}^m u_i(t) \,f_i(q(t)),\qquad \quad a.e.\,t\in[0,T]
	\end{equation}
	where $u:[0,T]\to\real^m$ is an integrable and bounded control function and $\vett$ is a family of smooth vector fields on $M$. We let $f_u=\sum_{i=1}^mu_i\,f_i$.
	The value function $\dsrs$ associated with the $\xLone$ cost is in fact a distance, called \emph{Carnot-Carath\'eodory } (or \emph{sub-Riemannian}) \emph{distance}. 
	Namely, for any $q,q'\in M$,
	\[
		\dsrs(q,q') =\inf \int_0^T\sqrt{\sum_{j=1}^m u_j(t)^2}\,dt,
	\]
	where the infimum is taken among all the controls $u\in\lcont$, for some $T>0$, such that its trajectory in \eqref{cs:sr} is such that $q_u(0)=q$ and $q_u(T)=q'$.
	An absolutely continuous curve $\gamma:[0,T]\to M$ is admissible for \eqref{cs:sr} if there exists $u\in \lcont$ such that $\dot\gamma(t)=f_u(t)$.

	Let $\distr$ be the $\xCinfty$-module generated by the vector fields $\vett$ (in particular, it is closed under multiplication by $\xCinfty(M)$ functions and summation). Let $\distr^1=\distr$, and define recursively $\distr^{s+1}= \distr^s+[\distr^s,\distr]$, for every $s\in\nat$. Due to the Jacobi identity $\distr^s$ is the $\xCinfty$-module of linear combinations of all commutators of $f_1,\ldots,f_m$ with length $\le s$. 
	For $q\in M$, let $\distr^s(q)=\{ f(q)\colon\: f\in\distr^s\}\subset T_q M$. We say that $\{f_1,\ldots,f_m\}$ satisfies the \em H\"ormander condition \em (or that it is a bracket-generating family of vector fields) if $\bigcup_{s\ge1} \distr^s(q)=T_q M$ for any $q\in M$. 
	Moreover, $\vett$ defines an \emph{equiregular} sub-Riemannian structure if $\dim\distr^i(q)$ does not depend on the point for any $i\in\nat$.
	In the following we will always assume these two conditions to be satisfied. 

	By the Chow--Rashevsky theorem (see for instance \cite{Agrachev2010}), the hypothesis of connectedness of $M$ and the H\"ormander condition guarantee the finiteness and  continuity of $\dsrs$ with respect to the topology of $M$. 
	Hence, the sub-Riemannian distance, induces on $M$ a metric space structure. The open balls of radius $\eps>0$ and centered at $q\in M$, with respect to $\dsrs$, are denoted by $\bsr (q, \eps)$. 

	We say that a control $u\in\lcont$, $T>0$, is a minimizer of the sub-Riemannian distance between $q,q'\in M$ if the associated trajectory $q_u$ with $q_u(0)=q$ is such that $q_u(T)=q'$ and $\|u\|_{\lcont}=\dsrs(q,q')$. 
	Equivalently, $u$ is a minimizer between $q,q'\in M$ if it is a solution of the free-time optimal control problem, associated with \eqref{cs:sr},
	\begin{equation}\label{eq:l1}
		\|u\|_{\xLone(0,T)}=\int_0^T \sqrt{\sum_{j=1}^m u_j^2(t)}\,dt\rightarrow \min,\qquad q_u(0)=q,\quad q_u(T)=q',\quad T>0.
	\end{equation}
	It is a classical result that, for any couple of points $q,q'\in M$ sufficiently close, there exists at least one minimizer.

	\begin{rmk}
	This control theoretical setting can be stated in purely geometric terms even if we drop the equiregularity assumption. Indeed, it is equivalent to a \em generalized sub-Riemannian \em structure. Such a structure is defined by a rank-varying smooth distribution and a Riemannian metric on it (see \cite{Agrachev2010} for a precise definition). 
	
	In a sub-Riemannian control system, in fact, the map $q\mapsto \Span\{f_1(q),\ldots,f_m(q)\}\subset T_qM$ defines a rank-varying smooth distribution, which is naturally endowed with the Riemannian norm defined, for $v\in \distr(q)$, by
	\begin{equation}
		\label{def:norm}
		{\metr}(q,v)=\inf \left\{ |u|=\sqrt{ u_1^2+\cdots+u_m^2} \colon\: f_u (q)= v \right\}.
	\end{equation}
	The pair $(\distr,\metr)$ is thus a generalized sub-Riemannian structure on $M$. Conversely, every rank-varying distribution is finitely generated, see \cite{Agrachev2008,Agrachev2010,Agrachev2010b,Drager2011}, and thus a sub-Riemannian distance can be written, globally, as the value function of a control system of the type \eqref{cs:sr}.
	\end{rmk}

	Since $\{f_1,\ldots,f_m\}$ is bracket-generating, the values of the sets $\distr^s$ at $q$ form a flag of subspaces of $T_q M$,
	\[
		\distr^1(q)\subset \distr^2(q) \subset\ldots \subset \distr^{r}(q)=T_qM.
	\]
	The integer $r$, which is the minimum number of brackets required to recover the whole $T_qM$, is called \em degree of non-holonomy \em (or \em step\em) of the family $\vett$ at $q$. 
	The degree of non-holonomy is independent of $q$ since we assumed the family $\vett$ to define an equiregular sub-Riemannian structure.
	Let $n_s=\dim \distr^s(q)$ for any $q\in M$. 
	The integer list $(n_1,\ldots,n_{r})$ is called the \em growth vector \em associated with \eqref{cs:sr}. 
	Finally, let $w_1\le\ldots\le w_n$ be the \em weights \em associated with the flag, defined by $w_i=s$ if $n_{s-1}<i\le n_s$, setting $n_0=0$. 
	 
	 For any smooth vector field $f$, we denote its action, as a derivation on smooth functions, by ${f: a\in \xCinfty(M)\mapsto fa\in \xCinfty(M)}$. For any smooth function $a$ and every vector field $f$ with $f\not\equiv 0$ near $q$, their \em (non-holonomic) order \em at $q$ is
	\[
	\begin{split}
		\ord_{q}(a)&=\min\{ s\in\nat\colon\: \exists i_1,\ldots,i_s\in\{1,\ldots,m\} \text{ s.t. } (f_{i_1}\ldots f_{i_s}\,a)(q)\neq 0\},\\	
		\ord_{q}(f)&=\max\{ \sigma\in\integ \colon\: \ord_{q}(fa) \ge \sigma+\ord_{q}(a) \text{ for any } a\in \xCinfty(M)\}.
	\end{split}
	\]
	 In particular it can be proved that $\ord_q(a)\ge s$ if and only if $a(q')=\bigo(\dsrs(q',q))^s$.	 
	 	 
	\begin{defn}
		A \em system of privileged coordinates \em at $q$ for $\vett$ is a system of local coordinates $z=(z_1,\ldots,z_n)$ centered at $q$ and such that $\ord_q(z_i)=w_i$, $1\le i\le n$.
	\end{defn}

	Let $q\in M$.
	A set of vector fields $\{f_1,\ldots,f_n\}$ such that 
	\begin{equation}
		    \{f_1(q),\,\ldots, \, f_n(q)\} \text{ is a basis of }T_qM, \quad\text{ and }\quad
		    f_i\in\distr^{w_i} \text{ for }i=1,\,\ldots,\,n,
	\end{equation}
	is called an adapted frame at $q$.
	We remark that to any system of privileged coordinates $z$ at $q$ is associated a (non-unique) adapted frame at $q$ such that $\partial_{z_i}=z_*f_i(q)$ (i.e., privileged coordinates are always linearly adapted to the flag).

	For any ordering $\{i_1,\ldots,i_n\}$, the inverse of the local diffeomorphisms 
	\begin{gather*}
		(z_1,\ldots,z_n)\mapsto e^{z_{i_1}\,f_{i_1}+\cdots+z_{i_n}\,f_{i_n}}(q),\qquad
		(z_1,\ldots,z_n)\mapsto e^{z_{i_n}\,f_{i_n}}\circ \cdots \circ e^{z_{i_1}\,f_{i_1}}(q),
	\end{gather*}
	defines privileged coordinates at $q$, called \em canonical coordinates of the first kind \em and of the \em second kind\em, respectively. 
	We remark that, for the canonical coordinates of the second kind, it holds $z_* f_{i_n}(z)\equiv\partial_{z_{i_n}}$.

	We recall the celebrated Ball-Box Theorem, that gives a rough description of the shape of small sub-Riemannian balls.
		
	\begin{thm}[Ball-Box Theorem]\label{thm:srballbox} 
	Let $z=(z_1,\ldots,z_n)$ be a system of privileged coordinates at $q\in M$ for $\vett$. Then there exist $C,\eps_0>0$ such that for any $\eps<\eps_0$, it holds
	\[
		\bbox {\frac 1 C \eps} \subset \bsr ({q}, \eps) \subset \bbox {C \eps},
	\] 
	where, $\bsr (q, \eps)$ is identified with its coordinate representation $z(\bsr (q, \eps))$ and,        for any $\eta>0$, we let
	\begin{equation}\label{def:box}
		\bbox \eta = \{z\in\real^n\colon\: |z_i|\le \eta^{w_i} \},
	\end{equation}

	\end{thm}	
	
	\begin{rmk}\label{rmk:bbuniform}
	Let $N\subset M$ be compact and let $\{z^q\}_{q\in N}$ be a family of systems of privileged coordinates at $q$ depending continuously on $q$.	    
	Then there exist uniform constants $C,\eps_0>0$ such that the Ball-Box Theorem holds for any $q\in N$ in the system $z^q$.
	\end{rmk}
	
\subsection{Control-affine systems}
\label{sub:control-affine_systems}

Let $\drift$ and $\vett$ be smooth vector fields on $M$ and, for some $\maxtime>0$, define $\contset=\bigcup_{0< T < \maxtime} \lcont$. Consider the control-affine control system
\begin{equation}\tag{D}\label{cs:d}
	\dot q(t) = \drift(q(t))+\sum_{j=1}^m u_j\,f_j(q(t)),\quad u\in\contset.
\end{equation}

An absolutely continuous curve $\gamma:[0,T]\to M$ is admissible for \eqref{cs:d} if $\dot\gamma(t)=\drift(\gamma(t))+f_{u(t)}(\gamma(t))$ for some control $u\in\lcont$. 
Observe, however, that contrary to what happens in the sub-Riemannian case, the admissibility for \eqref{cs:d} is not invariant under time reparametrization, e.g., a time reversal.
Thus there is no canonical choice for the cost, and we will focus on the two costs given in \eqref{int:costfunctions}. 

In the rest of the paper we will always assume the following hypotheses to be satisfied.
\begin{enumerate}
\renewcommand{\labelenumi}{(H\arabic{enumi})}
\renewcommand{\theenumi}{(H\arabic{enumi})}
	\item \label{hyp:1}	\em Equiregularity: \em $\dim\Delta^k(q)$ does not depend on $q\in M$;
	\item \label{hyp:2} \em Strong H\"ormander condition: \em there exists $r\in\nat$ such that $\Delta^r(q)=T_qM$ for any $q\in M$;
\end{enumerate}
Hypotheses~\ref{hyp:1} is made for technical reasons and to lighten the notation. 
It would be possible to avoid it through a desingularization procedure similar to the one in \cite{Jean2001a}.
On the other hand, \ref{hyp:2} is essential to apply our methods.

We will also often assume that, for some $s\in\nat$, the following ``equiregularity" for \eqref{cs:d} holds.
\begin{equation}\tag{H$_s$}\label{hyp:h0}
	\drift\subset\distr^s\setminus\distr^{s-1}.
\end{equation}
Due to hypothesis~\ref{hyp:1}, this is equivalent to the fact that $\ord_q(\drift)=-s$ for any $q\in M$.
	
For any $u\in\lcont$, by the variation formula (see \cite{Sachkov2004}), it holds
\begin{equation}\label{eq:split}
	\crex_0^T \bigg( \drift+\sum_{i=1}^m u_i(t)\, f_i \bigg)\,dt  = e^{T\drift}\circ \crex_0^T  \sum_{i=1}^m u_i(t)\, (e^{-t\drift})_* f_i\,dt.
\end{equation}
This shows that a control steering system \eqref{cs:d} from $p\in M$ to $q\in M$ in time $T>0$, steers from $p$ to $e^{-T\drift}q$ the time-dependent control system
\begin{equation}\tag{TD}\label{cs:td}
	\dot q(t) = \sum_{j=1}^m u_j(t)\,(e^{-t\drift})_*f_j(q(t)).
\end{equation}
Sometimes proofs will be eased by considering \eqref{cs:td} instead of \eqref{cs:d}, due to the linearity w.r.t.\ the control of the former.

In the following we will often consider also the two sub-Riemannian control systems associated with \eqref{cs:d}, called respectively \em small \em and \em big, \em and defined as
\begin{gather}
	\tag{SR-s}\label{cs:srs}
		\dot q(t) = \sum_{j=1}^m u_j(t)\,f_j(q(t)), \\
	\tag{SR-b}\label{cs:srb}
		\dot q(t) = u_0(t)\drift(q(t))+\sum_{j=1}^m u_j(t)\,f_j(q(t)).
\end{gather}
We will denote by $\dsrs$ and $\bsr$ the Carnot-Carath\'eodory metric and metric balls, respectively, associated with \eqref{cs:srs}. 
This distance is well-defined due to Hypothesis~\ref{hyp:2}

\section{Continuous families of coordinates} % (fold)
\label{sec:normal_forms}

In this section we consider properties of families of coordinates depending continuously on the points of the curve or path, in order to be able to exploit Remark~\ref{rmk:bbuniform}.

From the definition of privileged coordinates, we immediately get the following.
\begin{prop}
	\label{prop:privilegedproperty}
	Let $\gamma:[0,T]\to M$ be a path.
	Let $t>0$ and let $z$ be a system of privileged coordinates at $\gamma(t)$ for $\vett$.
    Then, there exists $C>0$ such that
    \begin{equation}
    	\label{eq:stimawjlep}
    	|z_j(\gamma(t+\xi))|\le C|\xi| \quad \text{for any } j=1,\ldots,n \text{ and any }t+\xi\in[0,T]. 
    \end{equation}
    Moreover, if for $k\in\nat$ it holds that $\dot\gamma(t)\notin\Delta^{k-1}(\gamma(t))$, then there exist $C_1,C_2,\xi_0>0$ and a coordinate $z_\alpha$, of weight $\ge k$, such that for any $t\in[0,T]$ and any $|\xi|\le \xi_0$ with $t+\xi\in[0,T]$ it holds
	\begin{equation}
		\label{eq:stimacrescita}
		C_1 \xi \le z_\alpha(\gamma(t+\xi)) \le C_2 \xi.
	\end{equation}
	Finally, if $\dot\gamma(t)\in \distr^k(\gamma(t))\setminus \Delta^{k-1}(\gamma(t))$, the coordinate $z_\alpha$ can be chosen to be of weight $k$.
\end{prop}

\begin{proof}
    By the smoothness of $\gamma$, there exists a constant $C>0$ such that $|(z_j)_* \dot\gamma(t+\xi)| \le C$ for any $j=1,\ldots,n$ and any $t+\xi\in[0,T]$.
    Thus, we obtain
    \[
    	|z_j(\gamma(t+\xi))| \le \left| \int_t^{t+\xi} |(z_j)_* \dot\gamma(t+\eta)|\,d\eta \right| \le C\,|\xi|.
    \]
    
    Let us prove \eqref{eq:stimacrescita}.
    Let $\{f_1,\ldots,f_n\}$ be an adapted basis associated with the system of coordinates $z$.
    In particular it holds that $z_*f_i(\gamma(t))=\partial_{z_i}$.
    Moreover, let $k'\ge k$ be such that $\dot\gamma(t)\in\distr^{k'}(\gamma(t))\setminus\distr^{k'-1}(\gamma(t))$ and write $\dot\gamma(t)=\sum_{w_i\le k'} a_i(t) f_{i}(\gamma(t))$ for some $a_i\in\xCinfty([0,T])$.
    Hence
    \[
    	z_* \dot\gamma(t)=\sum_{w_i\le k'} a_i(t)\, z_* f_{i}(\gamma(t)) = \sum_{w_i\le k'} a_i(t) \,\partial_{z_i}.
    \]
    Since there exists $i$ with $w_i=k'$ such that $a_i(t)\neq 0$, this implies that $(z_i)_* \dot\gamma(t)\neq0$.
    Since $k'\ge k$, we have then proved \eqref{eq:stimawjlep}.
\end{proof}

As already observed in Remark~\ref{rmk:bbuniform}, in order to be apply the estimates of Theorem~\ref{thm:srballbox} uniformly on $\gamma$ we need to consider a continuous family of coordinates $\{z^t\}_{t\in[0,T]}$ such that each $z^t$ is privileged at $\gamma(t)$ for $\vett$.
We will call such a family a \emph{continuous coordinate family for} $\gamma$.

Let us remark that, fixed any basis $\{f_1,\ldots,f_n\}$ adapted to the flag in a neighborhood of $\gamma([0,T])$, letting $z^t$ be the inverse of the diffeomorphism
\begin{equation}
	\label{eq:normcoordinatedrift}
	(z_1,\ldots,z_n)\mapsto e^{z_1 f_1}\circ\ldots \circ e^{z_{n} f_{n}}(\gamma(t)),
\end{equation}
defines a continuous coordinate family for $\gamma$.

The following proposition precises Proposition~\ref{prop:privilegedproperty}.

\begin{prop}
	\label{prop:normcoordpath}
	Let $\gamma:[0,T]\to M$ be a path and let $k\in\nat$ such that $\dot\gamma(s)\in\distr^k(\gamma(s))$ for any $t\in[0,T]$.
    Then, for any continuous coordinate family $\{z^t\}_{t\in[0,T]}$ for $\gamma$ there exists constants $C,\xi_0>0$ such that for any $t\in[0,T]$ and $0\le \xi \le\xi_0$ it holds
    \begin{equation}
    	\label{eq:stimak2}
    	|z_j^{t}(\gamma(t+\xi))|\le C \xi \quad \text{ if }w_j\le k \quad\text{ and }\quad
    	|z_j^{t}(\gamma(t+\xi))|\le C \xi^{\frac{w_j} k} \text{ if  }w_j>k. 
    \end{equation}
\end{prop}

\begin{proof}
	Fix $t\in[0,T]$ and let $\{f_1,\ldots,f_n\}$ be an adapted basis associated with the privileged coordinate system $z^t$.
	To lighten the notation, we do not explicitly write the dependence on time of such basis.
	Writing $z^t_*f_i(z)=\sum_{j=1}^n f_i^j(z)\partial_{z^t_j}$, it holds that $f_i^j$ is of weighted order $\ge w_j-w_i$, and hence there exists a constant $C>0$ such that 
	\begin{equation}
		\label{eq:ordine}
		|f_i^j(z)|\le C \|z\|^{(w_j-w_i)^+}.
	\end{equation}
	Here $\|z\|$ is the pseudo-norm $|z_1|^{\frac 1 w_1}+\cdots+|z_n|^{\frac 1 w_n}$ and $h^+=\max\{0,h\}$ for any $h\in\real$.
	Due to the compactness of $[0,T]$, the constant $C$ can be choosen to be uniform w.r.t.\ the time.

    Since $\dot\gamma(\xi)\in\distr^k(\gamma(\xi))$ for $\xi>0$, there exist functions $a_i\in\xCinfty([0,T])$ such that 
    \begin{equation}
    	\label{eq:dotgamma}
        \dot\gamma(\xi)=\sum_{w_i\le k} a_i(\xi) f_i(\gamma(\xi)) \qquad \text{ for any }\xi\in[0,T].
    \end{equation}    
    % Obviously, $a_\alpha(t_0)=1$ and $a_i(t_0)=0$ for any $i\neq0$.
    Observe that, for any $t\in[0,T]$, it holds that 
    \begin{equation}
    	\label{eq:ai}
    	\frac 1 \xi \int_t^{t+\xi} |a_i(\eta)|\,d\eta = |a_i(t)| + \bigo(\xi) \quad \text{as }\xi\downarrow0,
    \end{equation}
    where $\bigo(\xi)$ is uniform w.r.t. $t$.
    In particular, for any $\xi$ sufficiently small, this integral is bounded.

	By \eqref{eq:dotgamma}, for any $t\in[0,T]$ we get
	\begin{equation}
		\label{eq:dotgammacoord}
		z^t_j(\gamma(t+\xi))= \sum_{w_i\le k} \int_{t}^{t+\xi}a_i(\eta) f_i^j(z^t(\gamma(\eta)))\,d\eta,\qquad \text{for any }t+\xi\in[0,T]
	\end{equation}
	Then, applying \eqref{eq:ordine} we obtain
	\begin{equation}
		\label{eq:gammacoord}
		\begin{split}
			\max_{\rho\in[0,\xi]} |z^t_j(\gamma(t+\rho))|
				&\le\sum_{w_i\le k} \int_t^{t+\xi} |a_i(\eta)| \,|f_i^j(z^t(\gamma(\eta)))| \,d\eta\\
				&\le C\left(\max_{\rho\in[0,\xi]} \|z^t(\gamma(t+\rho))\|\right)^{(w_j-k)^+}  \sum_{w_i\le k}\int_t^{t+\xi} |a_i(\eta)| \,d\eta.
		\end{split}
	\end{equation}
	Up to enlarging the constant $C$, this and \eqref{eq:ai} yield
	\begin{equation}
		\label{eq:stimavarrho}
		\begin{split}
				\frac{\max_{\rho\in[0,\xi]} |z^t_j(\gamma(t+\rho))|}{\xi^{w_j}} 
					&\le C \left(\frac{\max_{\rho\in[0,\xi]}\|z^t(\gamma(t+\varrho^k))\|}{\xi}\right)^{(w_j-k)^+} \sum_{w_i\le k} \frac{1}{\xi^k} \int_t^{t+\xi^k} |a_i(\eta)|\,d\eta \\
					&\le C \left(\frac{\max_{\rho\in[0,\xi]}\|z^t(\gamma(t+\varrho^k))\|}{\xi}\right)^{(w_j-k)^+}.
			\end{split}
	\end{equation}

	Clearly, if ${\max_{\rho\in[0,\xi]}\|z^t(\gamma(t+\rho^k))\|}/{\xi}\le C$ uniformly in $t$, inequality \eqref{eq:stimavarrho} proves \eqref{eq:stimak2}.
	Then, let us assume by contradiction that ${\max_{\rho\in[0,\xi]}\|z^t(\gamma(t+\rho^k))\|}/{\xi}$ is unbounded as $\xi\downarrow 0$.
	For any $\xi$ let $\bar\xi\in[0,\xi]$ to be such that $\|z^t(\gamma(t+\bar\xi^k))\|=\max_{\rho\in[0,\xi]}\|z^t(\gamma(t+\rho^k))\|$.
	Then, there exists a sequence $\xi_\nu\rightarrow+\infty$ such that 
	\[
		b_\nu = \frac {|z_j^t(\gamma(t+\bar\xi_\nu^k))|} {\xi_\nu^{w_j}}\longrightarrow+\infty \quad\text{ and }\quad \frac 1 {n} \frac{\|z^t(\gamma(t+\bar\xi_\nu^k))\|}{\xi_\nu}\le b_\nu^{\frac 1 {w_j}}\le \frac{\|z^t(\gamma(t+\bar\xi_\nu^k))\|}{\xi_\nu}.
	\]
	Moreover, by \eqref{eq:stimavarrho}, it has to hold that $w_j>k$.
	Then, again by \eqref{eq:stimavarrho}, follows that
	\[
		{b_\nu} \le Cn\, b_\nu^{1-\frac k {w_j}} \longrightarrow 0 \quad \text{as }\nu\rightarrow+\infty.
	\]
	This contradicts the fact that $b_\nu\rightarrow+\infty$, and proves that there exists $\xi_0>0$, a priori depending on $t$, such that ${\|z^t(\gamma(t+\bar\xi^k))\|}/{\xi}\le C$ for any $\xi<\xi_0$.
	Since $[0,T]$ is compact, both constants $\xi_0,C$ are uniform for $t\in[0,T]$, thus completing the proof of \eqref{eq:stimak2} and of the proposition.
\end{proof}

We now focus on coordinate systems adapted to the drift.
In particular, if for some $s\in\nat$ it holds that $\drift\subset\distr^s\setminus\distr^{s-1}$, it makes sense to consider the following definition.

\begin{defn}
	\label{def:normdrift}
    A \emph{privileged coordinate system adapted to $\drift$ at $q$} is a system of privileged coordinates $z$ at $q$ for $\vett$ such that there exists a coordinate $z_\ell$ such that $z_* \drift\equiv\partial_{z_\ell}$.
\end{defn}

Observe that completing $\drift$ to an adapted basis $\{f_1,\ldots,\drift,\ldots,f_n\}$ allows us to consider the coordinate system adapted to $\drift$ at $q$, given by the inverse of the diffeomorphism
\begin{equation}
	\label{eq:normcoordinateonlydrift}
	(z_1,\ldots,z_n)\mapsto e^{z_\ell \drift}\circ\ldots \circ e^{z_{n} f_{n}}(q).
\end{equation}

The following definition combines continuous coordinate families for a path $\gamma:[0,T]\to M$ with coordinate systems adapted to a drift.

\begin{defn}
	\label{def:normadapted}
    A \emph{continuous coordinate family for $\gamma$ adapted to $\drift$} is a continuous coordinate family $\{z^t\}_{t\in[0,T]}$ for $\gamma$, such that each $z^t$ is a privileged coordinate system adapted to $\drift$ at $\gamma(t)$.
\end{defn}

Such coordinates systems can be built as per \eqref{eq:normcoordinateonlydrift}, letting the point $q$ vary on the curve.

Recall that $\drift\subset\distr^s\setminus\distr^{s-1}$ for some $s$, and consider a path $\gamma:[0,T]\to M$ such that $\dot\gamma(t)\in\distr^s(\gamma(t))$ and that $\drift(\gamma(t))\neq\dot\gamma(t)\mod\distr^{s-1}(\gamma(t))$ for any $t\in[0,T]$.
In this case, there exists $f_\alpha\subset\distr^s\setminus\distr^{s-1}$ and two functions $\varphi_\ell,\varphi_\alpha\in C^\infty([0,T])$, $\varphi_\alpha\ge 0$, such that 
\[
	\dot\gamma(t)\mod\distr^{s-1}(\gamma(t))=\varphi_\ell(t)\drift(\gamma(t))+\varphi_\alpha (t)f_\alpha(\gamma(t)).
\]
Moreover, by the assumption $\drift(\gamma(t))\neq\dot\gamma(t)\mod\distr^{s-1}(\gamma(t))$, if $\varphi_\ell(t)=1$ then $\varphi_\alpha(t)> 0$.
Then, using $f_\alpha$ as an element of the adapted basis used to define a continuous coordinate family for $\gamma$ adapted to $\drift$, it holds $(z^t_i)_*\dot\gamma( t)=\varphi_i(t)$ for $i=\alpha,\ell$ and any $t\in[0,T]$.
The following lemma will be essential to study this case.

\begin{lem}
	\label{lem:partizione}
	Assume that there exists $s\in\nat$ such that $\drift\subset\distr^s\setminus\distr^{s-1}$.
	Let $\gamma:[0,T]\to M$ be a path such that $\dot\gamma(t)\in\distr^s(\gamma(t))$ and such that  $\drift(\gamma(t))\neq\dot\gamma(t)\mod\distr^{s-1}(\gamma(t))$ for any $t\in[0,T]$. 
	Consider the continuous coordinate family $\{z^t\}_{t\in[0,T]}$ for $\gamma$ adapted to $\drift$ defined above. 
	Then, there exist constants $\xi_0,\rho,m>0$ and a coordinate $\alpha\neq \ell$ of weight $s$ such that for any $t\in[0,T]$ and $0\le \xi\le\xi_0$, it holds
	\begin{gather}
		\label{eq:prope1}
	    (z_\ell^t)_*\dot\gamma(t+\xi)\le 1-\rho \quad \text{if }t\in E_1=\{\varphi_\ell<1-2\rho\},\\
	    \label{eq:prope2}
	    (z_\alpha^t)_*\dot\gamma(t+\xi)\ge m \quad \text{if }t\in E_2=\{1-2\rho\le\varphi_\ell\le 1+2\rho\}, \\
	    \label{eq:prope3}
	    (z_\ell^t)_*\dot\gamma(t+\xi)\ge 1+\rho \quad \text{if }t\in E_3=\{\varphi_\ell>1+2\rho\}.
	\end{gather}
	In particular, it holds that $E_1\cup E_2\cup E_3=[0,T]$.
\end{lem}

\begin{proof}
	Since $\varphi_\alpha>0$ on ${\varphi_\ell^{-1}(1)}$, by continuity of $\varphi_\ell$ and $\varphi_\alpha$ there exists $\rho>0$ such that $\varphi_\alpha>0$ on ${\varphi_\ell^{-1}([1-2\rho,1+2\rho])}$.
	Since $E_2=\varphi_\ell^{-1}([1-2\rho,1+2\rho])$ is closed, letting $2m=\min_{E_2}\varphi_\alpha>0$
	property \eqref{eq:prope2} follows by the uniform continuity of $(t,\xi)\mapsto (z^t_\alpha)_*\dot\gamma(t+\xi)$ on ${E_2}\times[0,\xi_0]$, for sufficiently small $\xi_0$.
	Finally, the uniform continuity of $(t,\xi)\mapsto (z^t_\ell)_*\dot\gamma(t+\xi)$ over $\overline {E_1}\times[0,\xi_0]$ and $\overline {E_3}\times[0,\xi_0]$ yields \eqref{eq:prope1} and \eqref{eq:prope3}.
\end{proof}

We end this section by observing that when the path is well-behaved with respect to the sub-Riemannian structure, it is possible to construct a very special continuous coordinate family, rectifying both $\gamma$ and $\drift$ at the same time.

\begin{prop}\label{prop:normproperty}
   	Let $\gamma:[0,T]\to M$ be a path and $k\in\nat$ be such that $\dot\gamma(t)\in\distr^k(\gamma(t))\setminus\distr^{k-1}(\gamma(t))$ for any $t\in[0,T]$, there exists a continuous coordinate family  $\{z^t\}_{[0,T]}$ for $\gamma$ adapted such that 
    \begin{enumerate}
    	\item there exists a coordinate $z_\alpha$ of weight $k$ such that $z_*^t\dot \gamma \equiv \partial_{z_\alpha}$;
    	\item for any $\xi,\,t\in[0,T]$ it holds that $z_\alpha^t=z_\alpha^{t-\xi}+\xi$ and $z_i^t=z_i^\xi$ if $i\neq \alpha$.
    \end{enumerate}
    Moreover, if there exists $s\in\nat$ such that $\drift\subset\distr^s\setminus\distr^{s-1}$ and such that $\drift(\gamma(t))\neq\dot\gamma(t)\mod\distr^{s-1}(\gamma(t))$ for any $t\in[0,T]$ whenever $s=k$, such family can be chosen adapted to $\drift$.
\end{prop}

\begin{proof}
	By the assumptions on $\dot\gamma$, it is possible to choose $f_{\alpha}\subset\distr^k\setminus\distr^{k-1}$ such that $\dot\gamma(t)=f_{\alpha}(\gamma(t))$.
	Let then $\{f_1,\ldots,f_n\}$ be the adapted basis obtained by completing $f_\alpha$ and $\drift$.
	Finally, to complete the proof it is enough to consider the family of coordinates given by the inverse of the diffeomorphisms
	\[
		(z_1,\ldots,z_n)\mapsto e^{z_\ell \drift}\circ\cdots\circ e^{z_{\alpha} f_{\alpha}}(\gamma(t)).	\qedhere
	\]
\end{proof}

\section{Cost functions}
\label{sec:cost}

In this section we focus on properties of the cost functions defined in \eqref{int:costfunctions} and of the associated value functions, respectively denoted by $\val^\costuno(\cdot,\cdot)$ and $\val^\costdue(\cdot,\cdot)$.
For $\costuno$ such function is defined by
\begin{equation}\label{def:val}
	\val^\costuno(q,q')=\inf\left\{ \costuno(u,T) |\: T>0,\, q_u(0)=q,\, q_u(T)=q'   \right\}.
\end{equation}
The definition of $\val^\costdue$ is analogous.

\subsection{Regularity of the value function}

The following result, in the case of $\costuno$ is contained in \cite[Proposition 4.1]{prandi2013}, The proof can easily be extended to $\costdue$. 

\begin{thm}\label{thm:cont}
        For any $\maxtime>0$, the functions $\val^\costuno$ and $\val^\costdue$ are continuous from $M\times M\to [0,+\infty)$ (in particular they are finite). 
        Moreover, for any $q,q'\in M$ it holds
	\begin{gather*}
		\val^\costuno(q,q')\le \min_{0\le t\le\maxtime} \dsrs(e^{tf_0}q,q'),\\
		\val^\costdue(q,q')\le \min_{0\le t\le\maxtime} \big(t+ \dsrs(e^{tf_0}q,q') \big).
	\end{gather*}
	Here $e^{t\drift}$ denotes the flow of $\drift$ at time $t$ and $\dsrs$ denotes the Carnot-Carath\'eodory distance w.r.t.\ the system \eqref{cs:srs}, obtained from \eqref{cs:d} by putting $\drift=0$.
\end{thm}

We remark that this fact follows from the following proposition (obtained adapting \cite[Lemma 3.6]{prandi2013} to control-affine systems).

\begin{prop}\label{prop:costole}
    For any $\eta>0$ sufficiently small and for any $q_0,q_1\in M$, it holds 
 	\[
 		\inf\{ \costuno(u,\eta)\mid \text{ if } q_u(0)=q_0 \text{ then } q_u(\eta)=q_1  \}\le \dsrs(q_0,q_1).
 	\]
\end{prop}

We denote the reachable set from the point $q\in M$ with cost $\costuno$ less than $\eps>0$ as
\begin{equation}\label{def:reach}
	\bdt q \eps = \left\{ p\in M\mid \val^\costuno(q,p)\le \eps \right\}.
\end{equation}
Recall de definition of $\bbox \eta$ in \eqref{def:box} and that $\{\partial_{z_i}\}_{i=1}^n$ is the canonical basis in $\real^n$.
Then, we define the following sets, for parameters $\eta>0$ and $T>0$:
	\[
	\begin{split}
		\Xi_T(\eta) &= \bigcup_{0\le \xi\le T} \bigg(  \xi \partial_{z_\ell} + \bbox \eta  \bigg), \\ 	
		\Pi_T (\eta) &= \bbox \eta \cup \bigcup_{0<\xi\le T} \{ z\in\real^n\colon\: 0\le z_\ell-\xi\le \eta^s,\, |z_i|\le \eta^{w_i} + \eta \xi^{\frac {w_i} s} \text{ for } w_i\le s, i\neq \ell, \\
		&\qquad \qquad \qquad \qquad \qquad \qquad \qquad \qquad \qquad  \text{ and }|z_i|\le \eta (\eta + \xi^{\frac 1 s})^{w_i-1} \text{ for } w_i> s\}.\\
	\end{split}
	\]
In \cite{prandi2013} is proved a more general version of the following result, in the same spirit of Theorem~\ref{thm:srballbox}.

\begin{thm}\label{thm:bbox}
	Assume that there exists $s\in\nat$ such that $\drift\subset\distr^s\setminus\distr^{s-1}$.
	Assume, moreover, that $z=(z_1,\ldots,z_n)$ is a privileged coordinate system adapted to $\drift$, i.e., such that $z_* \drift=\partial_{z_\ell}$.
	Then, there exist $C,\eps_0,T_0>0$ such that 
	\begin{equation}\label{eq:bboxdrift}
		 \Xi_T\left( {\frac{1}{C} \eps}\right)\subset  \bdt q \eps \subset \Pi_T(C\eps),\qquad \text{for } \eps<\eps_0\text{ and } T<T_0.
	\end{equation}
	Here, with abuse of notation, we denoted by $\bdt q \eps$ the coordinate representation of the reachable set. 
	In particular,  
	\begin{equation}
		\label{eq:bboxdriftnegative}
			\bbox {\frac{1}{C} \eps}\cap \{z_\ell\le 0\}\subset \bdt q \eps \cap \{z_\ell\le 0\} \subset \bbox {C\eps}\cap \{z_\ell\le 0\}.
	\end{equation}
\end{thm}

\begin{rmk}\label{rmk:bbdriftuniform}
	Let $N\subset M$ be compact and let $\{z^q\}_{q\in N}$ be a family of systems of privileged coordinates at $q$ depending continuously on $q$.
	Then, as for Theorem~\ref{thm:srballbox} (see Remark~\ref{rmk:bbuniform}), there exist uniform constants $C,\eps_0,T_0>0$ such that Theorem~\ref{thm:bbox} holds for any $q\in N$ in the system $z^q$.
\end{rmk}

We notice also that, since \cite[Example 21]{prandi2013} is easily extendable to $\costdue$, it follows that, for neither $\costuno$ nor $\costdue$, the existence of minimizers is assured. 
Recall that a control $u\in\contset$ is a minimizer between $q_1,q_2\in M$ for the cost $\cost$ if its associated trajectory with initial condition $q_u(0)=q_1$ is such that $q_u(T)=q_2$ and $\val^\cost(q_1,q_2)=\cost(u,T)$. 

\subsection{Behavior along the drift}\label{subsec:drift}

The following proposition assures that a minimizer for $\costuno$ and $\costdue$ always exists when moving in the drift direction.

\begin{prop}\label{prop:opt}
	Assume that there exists $s\in \nat$ such that $\drift\subset\distr^s\setminus\distr^{s-1}$.
    For any $0<t<\maxtime$, the unique minimizer between any $q_0\in M$ and $e^{t\drift}q_0$ for the cost $\costuno$ is the null control on $[0,t]$.
	Moreover, if $\drift\notin\distr(q_0)$, i.e. $s\ge 2$, and the maximal time of definition of the controls $\maxtime$ is sufficiently small, the same is true for $\costdue$.
\end{prop}

\begin{proof}
	Since, for $t\in[0,\maxtime]$, we have that $\val^\costuno(q,e^{t\drift}q)=0$, the first statement is trivial.

	To prove the second part of the statement we proceed by contradiction.
	Namely, we assume that there exists a sequence $\maxtime_n\longrightarrow 0$ such that for any $n\in\nat$ there exists a control $v_n\in\xLone([0,t_n],\real^m)\subset\contsett {\maxtime_n}$, $v_n\not\equiv 0$ , steering the system from $q_0$ to $e^{\maxtime_n\drift}(q_0)$ and such that 
	\begin{equation}\label{eq:costoass}
		t_n+\|v_n\|_{\xLone([0,t_n],\real^m)}=\costdue(v_n,t_n) \le \costdue(0,\maxtime_n)=\maxtime_n.
	\end{equation}
	Let $z=(z_1,\ldots,z_n)$ be a privileged coordinate system adapted to $\drift$ at $q$, as per Definition~\ref{def:normdrift}.
	Thus, by Theorem~\ref{thm:bbox}, it holds
	\begin{equation}\label{eq:costoass2}
		\begin{split}
		    |z_\ell(e^{\maxtime_n\drift}(q_0))| \le t_n + C \|v_n\|_{\xLone([0,t_n],\real^m)}^2.
		\end{split}
	\end{equation}
	Since $z_\ell(e^{\maxtime_n\drift}(q_0))=\maxtime_n$, putting together \eqref{eq:costoass} and \eqref{eq:costoass2}  yields $\|v_n\|_{\xLone([0,t_n],\real^m)}\le C\|v_n\|_{\xLone([0,t_n],\real^m)}^2$ for any $n\in\nat$.
	Since by the continuity of $\val^\costdue$ we have that $\|v_n\|_{\xLone([0,t_n],\real^m)}\rightarrow 0$, this is a contradiction.
\end{proof}

We remark that, in the case of $\costdue$, the assumption $\drift\notin\distr(q_0)$ of Proposition~\ref{prop:opt} is essential.
In particular, in the following example we show that when $\drift\subset\distr$ even if a minimizer between $q_0$ and $e^{t\drift}(q_0)$ exists, it could not coincide with an integral curve of the drift.

\begin{example}
    Consider the control-affine system on $\real^2$,
    \begin{equation}
    	\label{eq:riemann}
    	\frac d {dt} x = \drift(x)+u_1 \drift(x)+u_2 f(x),
    \end{equation}
    where $\drift=(1,0)$ and $f=(\phi_1,\phi_2)$ for some $\phi_1,\phi_2:\real^2\to \real$, with $\phi_2\neq 0$ and $\partial_x(\phi_1/\phi_2)|_{(0,0)}\neq0$.
    Since $\drift$ and $f$ are always linearly independent, the underlying small sub-Riemannian system is indeed Riemannian with metric
    \[
    	g=
    	\left(
    	\begin{array}{c c}
    		1 & -{\phi_1}/{\phi_2} \\
    		-{\phi_1}/{\phi_2} & \frac{1-\phi_1^2}{\phi_2^{2}}
    	\end{array}
    	\right).
    \]

    Let us now prove that the curve $\gamma:[0,1]\to \real^2$, $\gamma(t)=(t\,T,0)$ is not a minimizer of the Riemannian distance between $(0,0)$ and $(T,0)$. 
    In particular, it is enough to prove that $\gamma$ is not a geodesic for small $T>0$. 
    For $\gamma$ the geodesic equation writes
    \[
    	\begin{cases}
    	    t^2 \Gamma_{11}^1(\gamma(t)) =0, \\
    	    t^2 \Gamma_{11}^2(\gamma(t)) =0,
    	\end{cases}
    	\quad
    	\text{for any }t\in[0,1]
    	\qquad
    	\Longleftrightarrow
    	\qquad
    	\Gamma_{11}^1(\cdot,0) =\Gamma_{11}^2(\cdot,0)=0 \text{ near }0.
    \]
    Here, $\Gamma_{k\ell}^i$ are the Christoffel numbers of the second kind associated with $g$.
    A simple computation shows that
    \[
    	\Gamma_{11}^1 = \frac {\phi_1} {\phi_2 }\, \partial_{x_1}\left(\frac {\phi_1} {\phi_2}\right),\quad 
    	\Gamma_{11}^2 =  \partial_{x_1}\left(\frac {\phi_1} {\phi_2}\right).
    \]
    Thus, if $\partial_{x_1} (\phi_1/\phi_2)|_{(0,0)}\neq 0$, then $\Gamma_{11}^2(0,0)\neq0$, showing that $\gamma$ is not a geodesic.
    
    We now show that this fact implies that for any minimizing sequence $u_n=(u_n^{1},u_n^2)\in\xLone([0,t_n],\real^2$  for $\val^\costdue$ between $(0,0)$ and $e^{T\drift}((0,0))=(T,0)$, such that $J(u_{n+1},t_{n+1})\le J(u_{n},t_{n})$, then $u_n^2\neq 0$ for sufficiently big $n$.
    To this aim, fix any $t_n\rightarrow 0$, let $u_n(s)=u(s/t_n)$ and $q_n(\cdot)$ be the trajectory associated with $u_n$ in system \eqref{eq:riemann}.
    Moreover, let $v=(v_1,0)\in\xLone([0,S],\real^2)$ be the minimizer of $\costdue$ between $(0,0)$ and $(T,0)$ in the system $\dot x_1=1+v_1$.
    Since the trajectory of $v$ is exactly $\gamma$, by rescaling it holds $\length(\gamma)=\costdue(v,S)$.
    Then, by standard results in the theory of ordinary differential equations, it follows that $q_n(t_n)\rightarrow (T,0)$ and the fact that $\gamma$ is not a Riemannian minimizing curve implies that
    \[
    \|u_n\|_{\xLone}=\|u\|_{\xLone}< \length(\gamma)=\costdue(v,S).
    \]
    Hence, for sufficiently big $n$ it holds that $\costdue(u_n,t_n)<\costdue(v,S)$, proving the claim.
\end{example}

As a consequence of Proposition~\ref{prop:opt}, we get the following property for the complexities defined in the previous section with respect to the costs $\costuno$ and $\costdue$. 
It generalizes to the control-affine setting the trivial minimality of the sub-Riemannian complexity on the path $\Gamma=\{q\}$.

\begin{cor}
	Assume that there exists $s\ge 2$ such that $\drift\subset\distr^s\setminus\distr^{s-1}$.
	Let $x\in M$ and $y=e^{T\drift}x$, for some $0<T<\maxtime$. 
	Then, for any $\eps>0$, the minimum over all curves $\Gamma\subset M$ (resp.\ paths $\gamma:[0,T]\to M$) connecting $x$ and $y$ of $\ccostcostuno(\cdot,\eps)$ and $\cappcostuno(\cdot,\eps)$ (resp. $\ctimecostuno(\cdot,\delta)$ and $\cneigcostuno(\cdot,\eps)$) is attained at $\Gamma=\{ e^{t\drift}\}_{t\in[0,T]}$ (resp.\ at $\gamma(t)=e^{t\drift}x$). 
	Moreover, the same is true for the cost $\costdue$, whenever $\maxtime$ is sufficiently small.
\end{cor}

\subsection{Behavior transversally to the drift}

When we consider two points on different integral curves of the drift, it turns out that the two costs $\costuno$ and $\costdue$ are indeed equivalent, as proved in the following.

\begin{prop}\label{prop:eq}
	Assume that there exists $s\in\nat$ such that $\drift\subset\distr^s\setminus\distr^{s-1}$. 
	Let $q,\,q'\in M$ 
	be such that there exists a set of privileged coordinates adapted to $\drift$ at $q$. 
	Then, there exists $C,\eps_0,\maxtime>0$ such that, for any $u\in\contset$  such that, for some $T<\maxtime$, $q_u(T)=q'$ and $\costuno(u,T)<\eps_0$, it holds
	\[
		\costuno(u,T) \le \costdue(u,T) \le C \costuno(u,T).
	\]
\end{prop}

The proof of this fact relies on the following particular case of \cite[Lemma 25]{prandi2013}.

\begin{lem}\label{lem:stimatau}
	Assume that there exists $s\in\nat$ such that $\drift\subset\distr^s\setminus\distr^{s-1}$. 
	Let $q\in M$ and let $z=(z_1,\ldots,z_n)$ be a system of privileged coordinate system adapted to $\drift$ at $q$.
	Then, there exist $C,\eps_0,\maxtime>0$ such that, for any $u\in \contset$,  with $\costuno(u,T) < \eps_0$ for some $T<\maxtime$,  it holds
	\[
		T \le C\big( \costuno(u,T)^s + z_\ell(q_u(T))^+ \big).
	\] 
	Here, we let $\xi^+=\max\{ \xi,0\}$.
\end{lem}

This Lemma is crucial, since it allows to bound the time of definition of any control through its cost.
We now prove Proposition~\ref{prop:eq}.

\begin{proof}[Proof of Proposition~\ref{prop:eq}]
	The first inequality is trivial. The second one follows by applying Lemma~\ref{lem:stimatau}, and computing
	\[
		 \costdue(u,T) \le T + \costuno(u,T) \le (C\eps_0^{s-1} + 1)\costuno(u,T).
	\]	
\end{proof}

\section{First results on complexities}
\label{sec:first}

In this section we collect some first results regarding the various complexities we defined. 

Firstly, we prove a result on the behavior of complexities.
For all the complexities under considreation, except the interpolation by time complexity, such result will hold with respect to a generic cost function $J:\contset\to [0,+\infty)$, satisfying some weak hypotheses.
% Namely, for any complexity we have to assume:
% \begin{enumerate}[(C1)]
% 	\item\label{hyp:all} For any $q_1\in M$ and any $q_2\notin \{ e^{t\drift}q_1 \}_{t\in[0,\maxtime]}$, it holds $\val^\cost(q_1,q_2)>0$.
% \end{enumerate}
% Moreover, in the case of the interpolation by time complexity, we need also the following.
% \begin{enumerate}[(C1)]
% 	\setcounter{enumi}{1}
% 	\item\label{hyp:triang} For any $u\in\lcont$ there exists a constant such that, if $t_1,t_2\in[0,T]$, $t_1<t_2$, then 
% 	\[
% 		\cost(u|_{[t_1,t_2]}(\cdot+t_1),t_2-t_1)\le C\cost(u,T).
% 	\]
% 	\item\label{hyp:stimatau} For any path $\gamma:[0,T]\to M$ with $\dot\gamma(t)\neq\drift(\gamma(t))$ for any $t\in[0,T]$, there exist $\eta>0$ and an interval $I\subset [0,T]$, with $|I|>\eta$, such that
% 	\[
% 		V\big(\gamma(t_1^h),\gamma(t_2^h)\big){\longrightarrow} 0 \text{ as } h\downarrow0 \implies t_2^h-t_1^h {\longrightarrow} 0 \text{ as } h\downarrow0,
% 	\]
% 	whenever $t_1^h\in I$ and $t_2^h>t_1^h$ for any $h$ in a right neighborhood of zero.
% \end{enumerate}
% We remark that for $\costuno$ and $\costdue$ these assumptions are always satisfied.
% In particular, (C\ref{hyp:stimatau}) holds thanks to Lemma~\ref{lem:stimatau}. 

\begin{prop}
  \label{prop:complexplode}
	Assume that for any $q_1\in M$ and any $q_2\notin \{ e^{t\drift}q_1 \}_{t\in[0,\maxtime]}$, it holds $\val^\cost(q_1,q_2)>0$.
	Then, the following holds.
	\begin{enumerate}[i.]
	\item For any curve $\Gamma\subset M$ it holds the following.
		\begin{enumerate}
			\item If the maximal time of definition of the controls, $\maxtime$, is sufficiently small, then
			$\lim_{\eps\downarrow0}\ccost(\Gamma,\eps)=\lim_{\eps\downarrow0}\capp(\Gamma,\eps) = +\infty. $
			\item If $\Gamma$ is an admissible curve for \eqref{cs:d}, then $\eps\ccost(\Gamma,\eps)$ and $\eps\capp(\Gamma,\eps)$ are bounded from above, for any $\eps>0$.
		\end{enumerate}
		
	\item For any path $\gamma:[0,T]\to M$ it holds the following.
		\begin{enumerate}
        \item 
			If $\gamma$ is not a solution of  \eqref{cs:d}, $\lim_{\eps\downarrow0}\cneig(\gamma,\eps) = +\infty$.
        \item If the cost is either $\costuno$ or $\costdue$,  $\drift\subset\distr^s\setminus\distr^{s-1}$, $\dot\gamma(t)\subset\distr^k(\gamma(t))\setminus\distr^{k-1}(\gamma(t))$ and $\drift(\gamma(t))\neq\dot\gamma(t)\mod\distr^{s-1}(\gamma(t))$ for any $t\in[0,T]$, then $\lim_{\eps\downarrow0}\ctime(\gamma,\eps)=+\infty$ whenever $\maxtimectime<\eta$.
        \item If $\gamma$ is an admissible curve for \eqref{cs:d}, then $\eps{\ctime(\gamma,\eps)}$ and $\eps\cneig(\gamma,\eps)$ are bounded by above, for any $\delta,\eps>0$.
		\end{enumerate}
	\end{enumerate}
\end{prop}

\begin{proof}
	The last statement for curves and paths follows simply by considering the control whose trajectory is the curve or the path itself, which is always admissible regardless of $\eps$.
		% On the other hand, let $u$ be a control such that $q_u([0,T])=\Gamma$. Then, it is clear that $\ccost(\Gamma,\eps)\le {\cost(u,T)}/ \eps$, completing the proof of the proposition.

	We now prove the first statement for the interpolation by cost complexity of a curve $\Gamma$.
	The same reasonings will hold for $\capp$ and $\cneig$.
	Let $x,y$ be the two endpoints of $\Gamma$ and assume $\maxtime$ to be sufficiently small so that $\val^\cost(x,y)>0$.
	Then, the first statement follows from
	\[
		\lim_{\eps\downarrow0}\ccost(\Gamma,\eps)\ge V(x,y) \lim_{\eps\downarrow0}\frac 1 \eps =+\infty.
	\] 

	Consider now the interpolation by time complexity and proceed by contradiction.
	Namely, let us assume that there exists a constant $C>0$ such that $\ctime(\gamma,\eps)\le C$ for any $\eps>0$.
	Then, by definition of $\ctime$, this implies that for any $\eps>0$ there exists $\delta_\eps\in[CT/2,\maxtimectime)$ and a $\delta_\eps$-time interpolation $u_\eps\in\lcont$ such that $\delta_\eps\cost(u_\eps,T)\le \eps$.

    Firstly, observe that by Lemma~\ref{lem:stimatau} and the assumptions on $\drift$ and $\dot\gamma$, we obtain that there exist $\eta>0$ and an interval $I\subset [0,T]$, with $|I|>\eta$, such that
	\begin{equation}
      \label{eq:c3}
		V\big(\gamma(t_1^h),\gamma(t_2^h)\big){\longrightarrow} 0 \text{ as } h\downarrow0 \implies t_2^h-t_1^h {\longrightarrow} 0 \text{ as } h\downarrow0,
	\end{equation}
	whenever $t_1^h\in I$ and $t_2^h>t_1^h$ for any $h$ in a right neighborhood of zero.

	For any $\eps$, let $0=t_0^\eps<t_1^\eps<\ldots<t_{N_\eps}^\eps=T$ be a partition of $[0,T]$ such that $q_{u_\eps}(t_i^\eps)=\gamma(t_i^\eps)$ for any $i\in\{0,N_\eps\}$ and $t_i^\eps-t_{i-1}^\eps\le\delta_\eps$.
	It is clear that, up to removing some $t_i^\eps$'s, we can assume that $t_i^\eps-t_{i-1}^\eps\ge\delta_\eps/2\ge CT/4$.
	Let us fix, $\tau_1^\eps=t_{i_\eps}^\eps\in I$ for some index $i_\eps$ and $\tau_2^\eps=t_{i_\eps+1}^\eps$.
	Such $\tau_1^\eps$ always exists, since $|I|>\delta_0$.
	Since, by the definition of $u_\eps$ and the choice of the cost, follows that $\val^\cost(\gamma(\tau_1^\eps),\gamma(\tau_{2}^\eps))\rightarrow 0$ as $\eps\downarrow 0$ we obtain a contradiction.
    In fact, this implies that
	\[
		0= \lim_{\eps\downarrow 0} \big(\tau_2^\eps-\tau_1^{\eps}\big)\ge \frac{CT}4 >0.
	\]
	% Thus, 
	% \[
	% 	0<\val^\cost(x,y)\le\lim_{\eps\downarrow 0}\cost(u_\eps,T)\le \lim_{\eps\downarrow 0} \frac {2\eps } {CT}= 0,
	% \]
	% completing the proof.
\end{proof}

\begin{rmk}
  Result ii.b, regarding the interpolation by time complexity, holds for any cost satisfying the assumptions of Proposition~\ref{prop:complexplode}, such that for any path $\gamma$ it holds \eqref{eq:c3}, and that, for any $u\in\lcont$, there exists a constant such that, if $t_1,t_2\in[0,T]$, $t_1<t_2$, then 
	\[
		\cost(u|_{[t_1,t_2]}(\cdot+t_1),t_2-t_1)\le C\cost(u,T).
	\]
\end{rmk}

\begin{figure}[tb]
	\begin{center}
		%!tikz editor 1.0
%!tikz source begin
\begin{tikzpicture}
	[decoration={
      markings,
      mark=at position 1 with {\arrow[scale=1.5]{stealth}};
    }]
	
	%\draw[step=.5,help lines] (0,-2) grid (6,2);
	
	% gamma
	\begin{scope}
		\draw (.5,1) node[above,scale=.9] {$\gamma(\cdot)$};
		\draw (0,0) arc (180:0:.5 and 1);
		\draw (1,0) arc (-180:0:.5 and 1);
		\draw (2,0) arc (180:0:.5 and 1);
		\draw (3,0) arc (-180:0:.5 and 1);
		\draw (4,0) arc (180:0:.5 and 1);
		\draw (5,0) arc (-180:0:.5 and 1);
	\end{scope}
	
	\node[below] at (0,0) {$x$};
	
	\node[above] at (6,0) {$y$};
	
	\node[below right,scale=.9] at (2.3,-1) {$e^{j \frac T N f_0}(x)$};
	\draw[->] (2.5,-1.1) -- (2.1,-.1);
	
	\foreach \x in {0,...,6}  
		{
		\filldraw[black] (\x,0) circle (1pt); 
		} 
	
	% null control
	\draw (0,0) -- (6,0);
	\draw (3.5,1) node[above,scale=.9] {$e^{\cdot f_0}(x)$};
	\draw[->] (3.5,1) -- (3.5,.1);
	
	%drift direction
	\draw[postaction={decorate}] (.2,-2) -- (1.3,-2) node[midway, above] {$f_0$};
\end{tikzpicture}
%!tikz source end
	\end{center}
	\caption{An example of a curve satisfying Remark~\ref{rmk:timeboundmot}, with a rectified drift.}
	\label{fig:maxtimectime}
\end{figure}
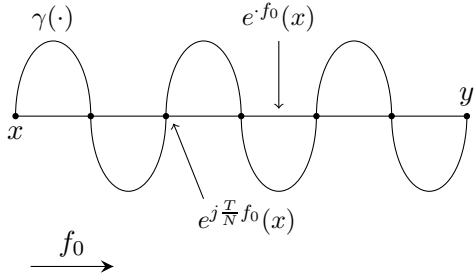

\begin{rmk}
   	\label{rmk:timeboundmot}
   	The bound on $\maxtimectime$ in  Proposition~\ref{prop:complexplode} is essential.
   	For example, consider the cost $\costuno(u,T)=\|u\|_{\lcont}$, and a curve such that, for some $N\in\nat$, it holds $\gamma(j T/N)=e^{j ({T}/{N}) \drift}(\gamma(0))$ for any $j=1,\ldots,T/N$ (see, e.g., Figure~\ref{fig:maxtimectime}).
   	In this case, the null control is a $ (T/N)$-time interpolation of $\gamma$, with $\costuno(0,T)=0$.
   	In particular, if $\maxtimectime>T/N$, it holds $\ctime(\gamma,\eps) \le N$.
\end{rmk}

In the following, we will denote with an apex ``SR-s'' -- e.g. $\srsccost$ -- the complexities associated with the small sub-Riemannian system \eqref{cs:srs} defined at p.~\pageref{cs:srs}, and with an apex ``SR-b'', e.g. $\srbccost$, the ones associated with the big sub-Riemannian system \eqref{cs:srb}. 

We immediately get the following.

\begin{prop}\label{prop:srs}
	Let $\Gamma\subset M$ be a curve and $\gamma:[0,T]\to M$ be a path. 
	\renewcommand{\theenumi}{\roman{enumi}}
	\begin{enumerate}
		\item Any complexity relative to the cost $\costuno$ is smaller than the same complexity relative to $\costdue$.	 Namely, for any $\eps,\delta>0$, it holds		
		\begin{gather*}
			\ccostcostuno(\Gamma,\eps) \le \ccostcostdue(\Gamma,\eps), \quad  \cappcostuno(\Gamma,\eps) \le \cappcostdue(\Gamma,\eps), \\
			\ctimecostuno(\gamma,\eps) \le \ctimecostdue(\gamma,\eps),\quad \cneigcostuno(\gamma,\eps) \le \cneigcostdue(\gamma,\eps).
		\end{gather*}
		
		\item For any cost, the neighboring approximation complexity of some path is always bigger than the tubular approximation complexity of its support.
		Namely, for any $\gamma:[0,T]\to M$ and any $\eps>0$, it holds
		\[
			\cappcostuno(\gamma([0,T]),\eps)\le\cneig^{\costuno}(\gamma,\eps),\qquad \cappcostdue(\gamma([0,T]),\eps)\le\cneigcostdue(\gamma,\eps) 
		\]

		\item Any complexity relative to the cost $\costdue$ is bigger than the same complexity computed for the system \eqref{cs:srb}. Namely, for any $\eps,\delta>0$, it holds	
		\begin{gather*}
			\srbccost(\Gamma,\eps) \le \ccostcostdue(\Gamma,\eps), \quad  \srbcapp(\Gamma,\eps) \le \cappcostdue(\Gamma,\eps), \\
			\srbctime(\gamma,\eps) \le \ctimecostdue(\gamma,\eps),\quad \srbcneig(\gamma,\eps) \le \cneigcostdue(\gamma,\eps).
		\end{gather*}
		\item In the case of curves, the complexities relative to the cost $\costdue$ are always smaller than the same complexities computed for the system \eqref{cs:srs}. Namely, for any $\eps>0$ it holds
		\begin{equation*}
			\ccostcostdue(\Gamma,\eps) \le\srsccost(\Gamma,\eps),\quad \cappcostdue(\Gamma,\eps) \le\srscapp(\Gamma,\eps).
		\end{equation*}
	\end{enumerate}
\end{prop}

\begin{proof}
	The inequality in $(ii)$ is immediate, since any control admissible for the $\cneig(\gamma,\eps)$ is also admissible for $\capp(\gamma([0,T]),\eps)$.

	On the other hand, the inequalities in $(iii)$ between the complexities in \eqref{cs:srb} and the ones in \eqref{cs:d}, with cost $\costdue$, is a consequence of the fact that, for every control $u\in\contset$, the trajectory $q_u$ is admissible for \eqref{cs:srb} and associated with the control $u_0=(1,u):[0,T]\to\real^{m+1}$ with $\|u_0\|_{\xLone([0,T],\real^{m+1})}=\costdue(u,T)$.
	The inequalities in $(i)$ between the complexities in \eqref{cs:d} with respect to the different costs follows from the fact that $\costuno\le\costdue$.

	Finally, to complete  the proof of the proposition, observe that, by Theorem~\ref{thm:cont}, it holds that
	\[
		\val^\costdue(q,q')\le \dsrs(q,q'),\quad\text{ for any }q,\,q'\in M.
	\]
	This shows, in particular, that every $\eps$-cost interpolation for \eqref{cs:srs}, is an $\eps$-cost interpolation for \eqref{cs:d}, proving the statement regarding the cost interpolation complexity in $(iv)$. 
	The part concerning the tubular approximation follows in the same way.
\end{proof}
  
We conclude this section by proving an asymptotic equivalence for the complexities of a control-affine system in a very special case.
In particular, we will prove that if we cannot generate the direction of $\Gamma$ with an iterated bracket of $f_0$ and some $f_1,\ldots,f_m$, then the curve complexities for the systems \eqref{cs:d}, \eqref{cs:srs} and \eqref{cs:srb} behaves in the same way. 

Let $\mathcal L_{\drift}$ be the ideal of the Lie algebra $\Lie(f_0,f_1,\ldots,f_m)$ generated by the adjoint endomorphism $\text{ad}(\drift): f \mapsto \text{ad}(\drift)f=[\drift,f]$, $f\in\text{Vec}(M)$. Then the following holds.

\begin{prop}
	\label{prop:firstestimate}
	Assume that there exists $s\in\nat$ such that $\drift\subset\distr^s\setminus\distr^{s-1}$,
	and let $\Gamma\subset M$ be a curve such that there exists $k\in\nat$ for which $T\Gamma\subset \Delta^k\setminus\Delta^{k-1}$. Assume, moreover, that for any $q\in\Gamma$ it holds that $T_q\Gamma\not\subset \mathcal L_{\drift}(q)$.
	Then, for sufficiently small $\maxtime$,
	\begin{equation}\label{eq:banal}
		\ccostcostuno(\Gamma,\eps)\asymp \ccostcostdue(\Gamma,\eps) \asymp \cappcostuno(\Gamma,\eps) \asymp \cappcostdue(\Gamma,\eps) \asymp \frac 1 {\eps^k}.
	\end{equation}
\end{prop}

\begin{proof}
	By the fact that $T_q\Gamma\not\subset \mathcal L_{\drift}(q)$, follows that $T_q\Gamma\subset\Lie^k_q(f_0,f_1,\ldots,f_m)\setminus\Lie^{k-1}_q(f_0,f_1,\ldots,f_m)$.
	Thus, approximating $\Gamma$ in the big or in the small sub-Riemannian system is equivalent, and by Theorem~\ref{thm:srcomplexity} follows
	\begin{equation*}\label{eq:jean}
		\srsccost(\Gamma,\eps)\asymp \srbccost(\Gamma,\eps) \asymp\srscapp(\Gamma,\eps)\asymp\srbcapp(\Gamma,\eps) \asymp \frac 1 {\eps^k}.
	\end{equation*}
	The statement then follows by applying Proposition~\ref{prop:srs}.
\end{proof}

\begin{rmk}\label{rmk:trans}
  Observe that if $\drift\in\Delta$ in a neighborhood $U$ of $\Gamma$, it holds that $\Lie^k_q(f_0,f_1,\ldots,f_m)=\Delta^k(q)$ for any $q\in U$.
  Then, by the same argument as  above, we get that \eqref{eq:banal} holds.
  This shows that, where $\drift\subset\Delta$, the asymptotic behavior of complexities of curves  is the same as in the sub-Riemannian case. 
\end{rmk}

\section{Complexity of curves} % (fold)
\label{sec:complexity_for_curves_under_h_s_}

This section is devoted to prove the statement on curves of Theorem~\ref{thm:driftcomplexity}.
Namely, we will prove the following.

\begin{thm}
	\label{thm:curvegeneral}
	Assume that there exists $s\ge 2$ such that $\drift\subset\distr^s\setminus\distr^{s-1}$.
	Let $\Gamma\subset M$ be a curve and define $\kappa=\max\{ k \colon\: T_p\Gamma\in\Delta^k(p)\setminus\Delta^{k-1}(p) $ for some $p\in\Gamma\}$.
	Then, 
	if the maximal time of definition of the controls $\maxtime$ is small enough,
	\[
		\ccostcostuno(\Gamma,\eps)\asymp \ccostcostdue(\Gamma,\eps)\asymp \cappcostuno(\Gamma,\eps)\asymp \cappcostdue(\Gamma,\eps)\asymp \frac 1 {\eps^{\kappa}},
	\]
\end{thm}

Due to the fact that the value functions associated with the costs $\costuno$ and $\costdue$ are always smaller than the sub-Riemannian distance associated with system \eqref{cs:srs}, the $\preccurlyeq$ immediately follows from the results in \cite{Jean2003a}.

\begin{prop}
   	\label{prop:curveupperbound}
   	Let $\Gamma\subset M$ be a curve such that there exists $k\in \nat$ for which $T\Gamma\subset \Delta^k$.
   	Then,
   	\[
   		\ccostcostuno(\Gamma,\eps)\preccurlyeq \ccostcostdue(\Gamma,\eps)\preccurlyeq \frac 1 {\eps^k},\qquad \cappcostuno(\Gamma,\eps)\preccurlyeq \cappcostdue(\Gamma,\eps)\preccurlyeq \frac 1 {\eps^k}.
   	\]
\end{prop}

\begin{proof}
	By $(i)$ in Proposition~\ref{prop:srs}, follows that we only have to prove the upper bound for the complexities relative to the cost $\costdue$.
    Moreover, by the same proposition and \cite[Theorem~3.14]{Jean2003a}, follows immediately that $\ccostcostdue(\Gamma,\eps)$ and $ \cappcostdue(\Gamma,\eps) \preccurlyeq {\eps^{-k}}$, completing the proof of the proposition. 
\end{proof}

In order to prove $\succcurlyeq$, we will need to exploit a sub-additivity property of the complexities.
In order to have this property, it is necessary to exclude certain bad behaving points, called cusps.
Near these points, the value function behaves like the Euclidean distance does near algebraic cusps (e.g., $(0,0)$ for the curve $y=\sqrt{|x|}$ in $\real^2$).
In the sub-Riemannian context, they have been introduced in \cite{Jean2003a}.

\begin{defn}
	\label{defn:cusps}
    The point $q\in\Gamma$ is a \emph{cusp} for the cost $\cost$ if it is not an endpoint of $\Gamma$ and if, for every $c,\eta>0$, there exist two points $q_1,q_2\in \Gamma$ such that $q$ lies between $q_1$ and $q_2$, with $q_1$ before $q$ and $q_2$ after $q$  w.r.t.\ the orientation of $\Gamma$ (in particular $q\neq q_1,q_2$), $\val^\cost(q_1,q_2)\le \eta$ and $\val^\cost(q,q_2)\ge c\, \val(q_1,q_2)$.
\end{defn}

In \cite{Jean2003a} is proved that no curve has cusps in an equiregular sub-Riemannian stucture.
As the following example shows, the equiregularity alone is not enough for control-affine systems.

\begin{example}
	Consider the following vector fields on $\real^3$, with coordinates $(x,y,z)$,
	\[
		f_1(x,y,z)=\partial_x,\qquad f_2(x,y,z)=\partial_y+x\partial_z.
	\]
	Since $[f_1,f_2]=\partial_z$, $\{f_1,f_2\}$ is a bracket-generating family of vector fields. 
	The sub-Riemannian control system associated with $\{f_1,f_2\}$ on $\real^3$ corresponds to the Heisenberg group.
	
	Let now $f_0=\partial_z\subset\distr^2\setminus\distr$ be the drift, and let us consider the curve $\Gamma=\{ (t^2,0,t)\mid t\in(-\eta,\eta)\}$.
	Let $q=(0,0,0)$.
	Since $T_{q}\Gamma\notin \distr(q)$, by smoothness of $\Gamma$ and $\distr$, for $\eta$ sufficiently small $T\Gamma\subset \distr^2\setminus\distr$. 
	We now show that the point $q$ is indeed a cusp for the cost ${\costuno}$. % and $\costdue$.
	In fact, for any $\xi>0$ such that $2\xi<\maxtime$, it holds that the null control defined over time $[0,2\xi]$ steers the control affine system from $q_1=(\xi^2,0,-\xi)\in \Gamma$ to $q_2=(\xi^2,0,\xi)\in \Gamma$.
	Hence, by Proposition~\ref{prop:opt}, $V^\costuno(q_1,q_2)=0$.
	Moreover, since $q$ and $q_2$ are not on the same integral curve of the drift, $V^\costuno(q,q_2)>0=V^\costuno(q_1,q_2)$.
	This proves that $q$ is a cusp for $\costuno$.
\end{example}

The following proposition shows that cusps appear only where the drift becomes tangent to the curve at isolated points, as in the above example. 

\begin{prop}
  	\label{prop:cusps}
  	Assume that there exists $s\ge2$ such that $\drift\subset\distr^s\setminus\distr^{s-1}$.
  	Let $\Gamma\subset M$ be a curve such that $T\Gamma \subset\Delta^k\setminus\Delta^{k-1}$.
	Moreover, if $s=k$, let $\Gamma$ be such that either $\drift(p)\notin T_p\Gamma\oplus\Delta^{s-1}(p)$ for any $p\in\Gamma$ or  $\drift|_{\Gamma}\subset T\Gamma\oplus\Delta^{s-1}$.
  	Then $\Gamma$ has no cusps for the cost $\val^\costuno$.
\end{prop}

\begin{proof}
	If $\drift|_{\Gamma}\subset T_p\Gamma\oplus\Delta^{s-1}(p)$, the statement is a consequence of Proposition~\ref{prop:opt}.
	Hence, we assume that $\drift(p)\notin T_p\Gamma\oplus\Delta^{s-1}(p)$ for any $p\in\Gamma$.
    Let $\gamma:[0,\mathfrak T]\to M$ be a path parametrizing $\Gamma$ and consider the continuous coordinate family $\{z^t\}_{t\in[0,\mathfrak T]}$ adapted to $\drift$ given by Proposition~\ref{prop:normproperty}.
    In particular, it holds that  $z^t_*\dot\gamma(\cdot)\equiv\partial_{z_\alpha}$ for some coordinate $z_\alpha$ of weight $k$ and for any $t\in[0,\mathfrak T]$.
    We now fix any $t_0\in (0,\mathfrak T)$ and prove that $\gamma(t_0)$ is not a cusp.
    In fact, letting $\eta>0$ be sufficiently small, by Theorem~\ref{thm:bbox} and the fact that $z_\ell^{t}(\gamma(\cdot))\equiv 0$ we get
    \[
    	\begin{split}
    	    V^\costuno(\gamma(t_0),\gamma(t_0+\eta))&\le C\sum_{j=1}^n |z_j^{t_0}(\gamma(t_0+\eta))|^{\frac 1 {w_j}} =C|z_\alpha^{t_0}\gamma(t_0+\eta)|^{\frac 1 k}\\
    	    &=2C |z_\alpha^{t_0-\eta}(\gamma(t_0+\eta))|^{\frac 1 {k}}\le C V(\gamma(t_0-\eta),\gamma(t_0+\eta)).
    	\end{split}
    \]
    Letting $t_1=t_0-\eta$ and $t_2=t_0+\eta$, this proves that $V^\costuno(\gamma(t_0),\gamma(t_2))\le V^{\costuno}(\gamma(t_1),\gamma(t_2))$.
    By definition, this implies that $\gamma(t_0)$ is not a cusp, completing the proof of the proposition.
\end{proof}

Finally, we can prove the sub-additivity of the curve complexities.

\begin{prop}
	\label{prop:semicont}
    Let $\Gamma'\subset\Gamma\subset M$ be two curves.
    Then, if the endpoints of $\Gamma'$ are not cusps for the cost $\costuno$,
    there exists a constant $C>0$ such that for sufficiently small $\maxtime$ it holds
    \[	
    	\ccostcostuno(\Gamma',\eps) \preccurlyeq \ccostcostuno(\Gamma,\eps),\qquad \cappcostuno(\Gamma',\eps) \preccurlyeq \cappcostuno(\Gamma,\eps).
    \]
\end{prop}

\begin{proof}
	\emph{Cost interpolation complexity.}
	Let $u\in\lcont$ be a control admissible for $\ccostcostuno(\Gamma,\eps)$, and let $0=t_1<\ldots<t_N=T$ be such that $\|u\|_{\xLone([t_{i-1},t_i])}\le \eps$.
	Recall that by Theorem~\ref{thm:cont}, $V^\costuno$ is a continuous function.
	Since for small $\maxtime>0$, for any $\eps>0$ and for any $q_0\in M$ the reachable set $\mathcal R_{\maxtime} (q, \eps)$ is bounded, it holds that $\mathcal R_{\maxtime} (q, \eps) \searrow \{ e^{t\drift}(q_0)\mid t\in[0,\maxtime] \}$ as $\eps\downarrow 0$, in the sense of pointwise convergence of characteristic functions. 
	From this follows that, for $\eps$ and $\maxtime$ sufficiently small, there exist ${i_1}\neq i_2$ such that $q_u(t_{i})\in \Gamma'$ for any $i\in \{i_1,\ldots,i_2\}$ and $q_u(t_{i})\not\in\Gamma'$ for any $i\notin\{i_1,\ldots,i_2\}$.
	Since $x'$ and $y'$ are not cusps, there exists $c>0$ such that, letting $x'$ and $y'$ be the endpoints of $\Gamma'$, it holds $\val^\costuno(x',q_u(t_{i_1}))\le c\val^\costdue(q_u(t_{i_1-1},q_u(t_{i_1}))\le \eps$ and $\val^\costuno(q_u(t_{i_2}),y')\le\val^\costuno(q_u(t_{i_2}),q_u(t_{i_2+1}))\le c\eps$.
	Thus, there exists a constant $C>0$ such that
	\begin{equation*}
		\ccostcostuno(\Gamma',\eps)\le \frac {\costuno(u|_{[t_{i_1},t_{i_2}]})} \eps +2c\le C \frac {\costuno(u|_{[t_{i_1-1},t_{i_2+1}]})} \eps \le C \frac {\costuno(u)} \eps.
	\end{equation*}
	Taking the infimum over all controls $u$, admissible for $\ccostcostuno(\Gamma,\eps)$ completes the proof.

	\emph{Tubular approximation complexity.} 
	Let $u\in\lcont$ be a control admissible for $\cappcostuno(\Gamma,\eps)$.
	Then, letting $q_u$ be its trajectory such that $q_u(0)=x$, there exists two times $t_1$ and $t_2$ such that $q_u(t_1)\in\bsr(x',C\eps)$ and $q_u(t_2)\in\bsr(y',C\eps)$.
	Then, since  $V^\costuno\le \dsrs$ by Theorem~\ref{thm:cont}, the same argument as above applies.
\end{proof}

Thanks to the sub-additivity, we can prove the $\succcurlyeq$ part of Theorem~\ref{thm:curvegeneral} in the case where the curve is always tangent to the same stratum $\distr^k\setminus\distr^{k-1}$.

\begin{prop}
	\label{prop:curvelowerbound}
    Assume, that there exists $s\in\nat$ such that $\drift\subset\distr^s\setminus\distr^{s-1}$.
    Let $\Gamma\subset M$ be a curve such that there exists $k\in\nat$ for which $T_p\Gamma\in\distr^k(p)\setminus\distr^{k-1}(p)$ for any $p\in\Gamma$.
    Then, for sufficiently small time $\maxtime$, it holds
    \[
    	\ccostcostdue(\Gamma,\eps)\succcurlyeq\ccostcostuno(\Gamma,\eps)\succcurlyeq \frac 1 {\eps^k},\qquad \cappcostdue(\Gamma,\eps)\succcurlyeq\cappcostuno(\Gamma,\eps)\succcurlyeq \frac 1 {\eps^k}.
    \]
\end{prop}

\begin{proof}
    By Proposition~\ref{prop:srs}, $\ccostcostdue(\Gamma,\eps)\succcurlyeq\ccostcostuno(\Gamma,\eps)$ and $\cappcostdue(\Gamma,\eps)\succcurlyeq\cappcostuno(\Gamma,\eps)$.
    We will only prove that $\ccostcostuno(\Gamma,\eps)\succcurlyeq \eps^{-k}$, since the same arguments apply to $\cappcostuno(\Gamma,\eps)$.

    Let $\gamma:[0,\mathfrak T]\to M$ be a path parametrizing $\Gamma$.
    We will distinguish three cases.

	\begin{description}
	    \item[Case 1 $\drift(p)\notin \distr^{s-1}(p)\oplus T_p\Gamma$ for any $p\in\Gamma$]
		Fix $\eta>0$ and consider a control $u\in \lcont$, admissible for $\ccost(\Gamma,\eps)$ such that
		\begin{equation}\label{eq:pezzi}
			\frac{\|u\|_{\xLone}} \eps \le \ccost(\Gamma,\eps)+\eta.
		\end{equation}
		Let $u_i=u|_{[t_{i-1},t_i]}$, $i=1,\ldots, N=\left\lceil \frac{\|u\|_{\xLone}}{\eps} \right\rceil$ to be such that $\|u_i\|_{\xLone}=\eps$ for any $1\le i < N$, $\|u_N\|_{\xLone}\le \eps$.
		Moreover, let $s_i$ be the times such that $\gamma(s_i)=q_u(t_i)$.
		
		By \eqref{eq:pezzi}, it holds $N\le \lceil\ccost(\Gamma,\eps)+\eta+1\rceil$.
		However, we can assume w.l.o.g.\ that $N\le \lceil\ccost(\Gamma,\eps)+\eta\rceil$. 
		In fact, $N>\lceil\ccost(\Gamma,\eps)+\eta\rceil$ only if  $\|u_N\|<\eps$.
		In this case we can simply restrict ourselves to compute $\ccost(\tilde\Gamma,\eps)$ where $\tilde\Gamma$ is the segment of $\Gamma$ comprised between $x$ and $q_u(t_{N-1})$. 
		Indeed, by Propositions~\ref{prop:cusps} and \ref{prop:semicont}, it follows that $\ccost(\tilde\Gamma,\eps)\preccurlyeq \ccost(\Gamma,\eps)$.
		
		We now assume that $\eps$ and $\maxtime$ are sufficiently small, in order to satisfy the hypotheses of Theorem~\ref{thm:bbox} at any point of $\Gamma$. 
		Moreover, let $\{z^t\}_{t\in[0,\mathfrak T]}$ be the continuous coordinate family for $\Gamma$ adapted to $\drift$ given by Proposition~\ref{prop:normproperty}.
		Then, it holds 
		\begin{equation}\label{eq:varcompl}
				\mathfrak T = \sum_{i=1}^N (s_i-s_{i-1}) = \sum_{i=1}^N|z_\alpha^{s_{i-1}}(\gamma(s_i))| = \sum_{i=1}^N|z_\alpha^{s_{i-1}}(q_u(t_i))|\le C(\ccost(\Gamma,\eps)+\eta) \eps^k. \\
		\end{equation}
		Here, in the last inequality we applied Theorem~\ref{thm:bbox} and the fact that $z_\ell^{s_i-1}(q_u(t_i))=0$ by Proposition~\ref{prop:normproperty}.
		Finally, letting $\eta\downarrow 0$ in \eqref{eq:varcompl}, we get that for any $\eps$ sufficiently small it holds $\ccost(\Gamma,\eps)\ge C\mathfrak T \, \eps^{-k}$.
		This completes the proof in this case.

		\item[Case 2 $s = k$ and $\drift(p)\in \distr^{s-1}(p)\oplus T_p\Gamma$ for any $p\in\Gamma$]
		Let $\{z^t\}_{t\in[0,\mathfrak T]}$ be a continuous coordinate family for $\gamma$ adapted to $\drift$.
		In this case, since $(z^t_\ell)_*\drift=1$, it holds that $(z^t_\ell)_*\dot\gamma(\cdot)\neq 0$.
		Hence, there exist $C_1,\,C_2>0$ such that for any $t,\xi\in[0,T]$
		\begin{gather}
			\label{eq:zellpos} C_1(t-\xi)\le z_\ell^t(\gamma(\xi)) \le C_2 (t-\xi), \quad \text{if } (z^t_\ell)_*\dot\gamma(\cdot)>0;\\
			\label{eq:zellneg} C_1(t-\xi)\le -z_\ell^t(\gamma(\xi)) \le C_2 (t-\xi), \quad \text{if } (z^t_\ell)_*\dot\gamma(\cdot)<0.
		\end{gather}
		If \eqref{eq:zellneg} holds, then we can proceed as in Case~1 with $\alpha=\ell$. In fact, $|z_\ell^{s_{i-1}}(q_u(t_i)|\le C\eps^s$ by Theorem~\ref{thm:bbox}.
		On the other hand, if \eqref{eq:zellpos} holds, by applying Theorem~\ref{thm:bbox} we get
	    \[
	    	\begin{split}
	    		\mathfrak T &= \sum_{i=1}^N (s_i-s_{i-1}) \le\frac {1}{C_1} \sum_{i=1}^N | z_\ell^{s_{i-1}}(\gamma(s_i))|= \frac {1}{C_1}\sum_{i=1}^N |z_\ell^{s_{i-1}}(q_u(t_i))| \\
	    			&\le \frac {1}{C_1} \sum_{i=1}^N (C\eps^s+t_i-t_{i-1}) \le C\big(\ccostcostuno(\Gamma,\eps)+\eta \big)\eps^{s} + T.
	    	\end{split}
	    \]
	    By taking $\maxtime$ sufficiently small, it holds $T\le \maxtime<\mathfrak T$.
	    Then, letting $\eta\downarrow 0$ this proves that $\ccostcostuno(\Gamma,\eps)\ge ((T-\maxtime)/C ) \eps^{-s}\succcurlyeq \eps^{-s}$.
	    This completes the proof of this case.

	    \item[Case 3 $s = k$ and $\drift(p)\in \distr^{s-1}(p)\oplus T_p\Gamma$ for some $p\in\Gamma$]
	    In this case, there exists an open interval $(t_1,t_2)\subset [0,\mathfrak T]$ such that $\drift(\gamma(t))\neq\dot\gamma(t) \mod\Delta^{s-1}(\gamma(t))$ for any $t\in (t_1,t_2)$.
	    Thus, $\Gamma'=\gamma((t_1,t_2))$, satisfies the assumption of Case 1 and hence $\ccostcostuno(\Gamma',\eps)\succcurlyeq\eps^{-k}$.
	    Moreover, by Proposition~\ref{prop:cusps}, we can assume that $\gamma(t_1)$ and $\gamma(t_2)$ are not cusps.
	    Then, by Proposition~\ref{prop:semicont} we get
	    \[
	    	\frac 1 {\eps^k}\preccurlyeq \ccostcostuno(\Gamma',\eps)\preccurlyeq \ccostcostuno(\Gamma,\eps),
	    \]
	    completing the proof of the proposition.
	\end{description}
	\end{proof}

Finally, we are in a condition to prove the main theorem of this section.

\begin{proof}[Proof of Theorem~\ref{thm:curvegeneral}]
    Since it is clear that $T\Gamma\subset \Delta^\kappa$, the upper bound follows by Proposition~\ref{prop:curveupperbound}.
    Moreover, by Proposition~\ref{prop:srs} it suffices to prove that $\ccostcostuno(\Gamma,\eps)$ and $\cappcostuno(\Gamma,\eps)\succcurlyeq \eps^{-\kappa}$.
    Since the arguments are analogous, we only prove this for $\ccostcostuno$.

    By smoothness of $\Gamma$, the set $A=\{p\in\Gamma \mid T_p\Gamma\in\Delta^\kappa(p)\setminus\Delta^{\kappa-1}(p)\}$ has non-empty interior.
    Let then $\Gamma'\subset A$ be a non-trivial curve such that either $\drift(p)\notin T_p\Gamma'\oplus\Delta^{s-1}(p)$ for any $p\in\Gamma'$ or that $\drift|_{\Gamma'}\subset T\Gamma'\oplus\Delta^{s-1}$.
    Then, since by Proposition~\ref{prop:cusps} we can choose $\Gamma'$ such that it does not contain any cusps, applying Proposition~\ref{prop:semicont} yields that $\ccostcostuno(\Gamma',\eps)\preccurlyeq \ccostcostuno(\Gamma,\eps)$.
    Finally, the result follows from the fact that, by Proposition~\ref{prop:curvelowerbound}, it holds $\ccostcostuno(\Gamma',\eps)\succcurlyeq \eps^{-\kappa}$.
\end{proof}

% section complexity_for_curves_under_h_s_ (end)

\section{Complexity of paths} % (fold)
\label{sec:complexity_of_paths_under_h_s_}

In this section we will prove the statement on paths of Theorems~\ref{thm:srtimecomplexity} and \ref{thm:driftcomplexity}.

Recall the definition of $\delta$-time interpolation given in Section~\ref{sub:complexities}, and define the following function of a path $\gamma:[0,T]\to M$ and a time-step $\delta>0$
\begin{equation*}
	\ctimeaux(\gamma,\delta) = \delta \inf \big\{ J(u,T) |\: u \text{ is a }\delta\text{-time interpolation of }\gamma \big\}.
\end{equation*}
Controls admissible for the above infimum define trajectories touching $\gamma$ at intervals of time of length at most $\delta$.
Then, function $\ctimeaux(\gamma,\delta)$ measures the minimal average cost on each of these intervals.
It is possible to express the interpolation by time complexity through $\ctimeaux$.
Namely,
\begin{equation}
	\label{eq:interpbytime}
	\ctime(\gamma,\eps) = \inf_{\delta\le\delta_0} \left\{ \frac T \delta \bigg|\: \ctimeaux(\gamma,\delta)\le \eps \right\} = \sup_{\delta\le \delta_0} \left\{ \frac{T}{\delta} \bigg|\: \ctimeaux(\gamma,\delta')\ge\eps \text{ for any } \delta'\ge \delta \right\}.
\end{equation}
From \eqref{eq:interpbytime} follows immediately that, for any $k\in\nat$,
\begin{equation}
	\label{eq:complaux}
	\ctime(\gamma,\eps)\preccurlyeq \eps^{-k}  \iff  \ctimeaux(\gamma,\delta) \preccurlyeq \delta^{\frac{1}{k}} \quad\text{ and }\quad \ctime(\gamma,\eps)\succcurlyeq \eps^{-k}  \iff  \ctimeaux(\gamma,\delta) \succcurlyeq \delta^{\frac{1}{k}}.
\end{equation}
Exploiting this fact, we are able to prove Theorem~\ref{thm:srtimecomplexity}.

\begin{proof}[Proof of Theorem~\ref{thm:srtimecomplexity}]
	\label{pf:srtimecomplexity}
	Let $\{z^t\}_{t\in[0,T]}$ to be the continuous family of coordinates for $\gamma$  given by Proposition~\ref{prop:normproperty}.
	We start by proving that $\ctimeaux(\gamma,\delta)\preccurlyeq\delta^{\frac 1 k}$ which, by \eqref{eq:complaux}, will imply $\srsctime(\gamma,\eps)\preccurlyeq\eps^{-k}$.
	Fix any partition $0=t_0<t_1<\ldots<t_N=T$ such that $\delta/2 \le t_i-t_{i-1}\le \delta$.
	If $\delta$ is sufficiently small, from Theorem~\ref{thm:srballbox} follows that there exists a constant $C>0$ such that for any $i=0,\ldots,N$ in the coordinate system $z^{t_i}$ it holds that $\text{Box}({\gamma(t_i),C\delta^{\frac 1 k}}) \subset \bsr(\gamma(t_i),\delta^{\frac 1 k})$. 
	Hence, since $z_\alpha^{t_{i-1}}(\gamma(t_i))=t_i-t_{i-1}$, that $z_j^{t_{i-1}}(\gamma(t_i))=0$ for any $j\neq \alpha$, and that $N\le \lceil 2T/\delta\rceil\le C T/\delta$, we get
	\begin{equation*}
		\ctimeaux(\gamma,\delta)\le \delta\sum_{i=1}^N \dsrs(\gamma(t_{i-1}),\gamma(t_i)) \le C\delta \sum_{i=1}^N \sum_{j=1}^n |z_j^{t_{i-1}}(\gamma(t_i))|^{\frac 1 {w_j}} =C\delta \sum_{i=1}^N  (t_i-t_{i-1})^{\frac 1 {k}} \leq CT\delta^{\frac 1 k}.
	\end{equation*}
	This proves completes the proof of the first part of the Theorem.
	
	Conversely, to prove that $\ctime(\gamma,\eps)\preccurlyeq \eps^{-k}$ we need to show that $\ctimeaux(\gamma,\delta)\succcurlyeq\delta^{\frac 1 k}$.
	To this aim, let  $\eta>0$ and $u\in \xLone$ be a control admissible for $\ctimeaux(\gamma,\delta)$ such that 
	\[
		\|u\|_{\xLone([t_{i-1},t_i])}\le \frac{ \ctimeaux(\gamma,\delta) } \delta+ \eta.
	\]
	Let $0=t_0<t_1<\ldots<t_N=T$ be times such that $q_u(t_i)=\gamma(t_i)$, $i=0,\ldots,N$, $0<t_i-t_{i-1}\le \delta$.
	Moreover, let $u_i\in \xLone([t_{i-1},t_{i}])$ be the restriction of $u$ between $t_{i-1}$ and $t_i$. 
	Observe that, up to removing some $t_i$'s, we can assume that $t_i-t_{i-1}\in \left(\frac \delta 2, \frac 3 2 \delta  \right]$.
	This implies that $\left\lceil  {2T}/( {3\delta}) \right\rceil\le N\le \left\lceil  {2T}/ \delta \right\rceil$.
		
	To complete the proof it suffices to show that $\|u_i\|_{\xLone([t_{i-1},t_i])}\ge C \delta^{\frac 1 {k}}$. 
	In fact, for any $\eta>0$, this yields
	\[
	\begin{split}
		\frac{\ctimeaux(\gamma,\delta)} \delta &\ge \|u\|_{\lcont}-\eta =\sum_{i=1}^N \|u_i\|_{\xLone([t_{i-1},t_i])} -\eta  \ge C\, \sum_{i=1}^N  \delta ^{ \frac 1 {k} } -\eta 
		\ge C\,   \frac{2T} {3\delta} \delta ^{ \frac 1 {k} } -\eta.
	\end{split}
	\]
	Letting $\eta\downarrow 0$, this will prove that $\ctimeaux(\gamma,\delta)\succcurlyeq \delta^{\frac 1 k}$, completing the proof.
	
	Observe that, by Theorem~\ref{thm:srballbox}, for any $i=1,\ldots,N$ in the coordinate system $z^{t_{i-1}}$ it holds  $\bsr(\gamma(t_{i}),\|u_i\|_{\xLone([t_{i-1},t_i])})\subset \bbox{\gamma(t_{i}),C\|u_i\|_{\xLone([t_{i-1},t_i])}}$.
	Since $z_\alpha^{t_{i-1}}(t_i)=t_i-t_{i-1}$, this implies that
	\[
		\frac \delta 2 \le  t_{i}-t_{i-1} =  |z^{t_{i-1}}_\alpha(\gamma(t_{i}))| \le C \, \|u_i\|_{\xLone([t_{i-1},t_i])}^{k},
	\]
	proving the claim and the theorem.
\end{proof}

The rest of the section will be devoted to the proof of the statement on paths of Theorem~\ref{thm:driftcomplexity}.
Namely, we will prove the following.

\begin{thm}
	\label{thm:pathgeneral}
	Assume that there exists $s\ge 2$ such that $\drift\subset\distr^s\setminus\distr^{s-1}$.
	let $\gamma:[0,T]\to M$ be a path such that $\drift(\gamma(t))\neq \dot\gamma(t) \mod\Delta^{s-1}(\gamma(t))$ for any $t\in [0,T]$ and define
    $\kappa=\max\{ k \colon\: \gamma(t)\in\Delta^k(\gamma(t))\setminus\Delta^{k-1}(\gamma(t)) $ for any $t$ in an open subset of $[0,T]\}$. 
	Then, it holds
	\[
		\ctimecostuno(\gamma,\eps)\asymp \ctimecostdue(\gamma,\eps)\asymp \cneigcostuno(\gamma,\eps) \asymp \cneigcostdue(\gamma,\eps) \asymp \frac 1 {\eps^{\max\{\kappa,s\}}},
	\]
	where the asymptotic equivalences regarding the interpolation by time complexity are true only when $\maxtimectime$, i.e., the maximal time-step in $\ctime(\gamma,\eps)$, is sufficiently small.
\end{thm}

Differently to what happened for curves, the $\preccurlyeq$ part does not immediately follow from the estimates of sub-Riemannian complexities, but requires additional care.
It is contained in the following proposition.

\begin{prop}
   	\label{prop:pathupperbound}
   	Assume that there exists $s\in\nat$ such that $\drift\subset\distr^s\setminus\distr^{s-1}$.
   	Let $\gamma:[0,T]\to M$ be a path such that $\dot\gamma(t)\in \Delta^k(\gamma(t))$.
   	Then, it holds
   	\begin{equation}
   		\label{eq:compltimek}
   		\ctimecostuno(\gamma,\eps)\preccurlyeq \ctimecostdue(\gamma,\eps)\preccurlyeq \frac 1 {\eps^{\max\{s,k\}}},
   	   		\qquad 
   	   		\cneigcostuno(\Gamma,\eps)\preccurlyeq \cneigcostdue(\Gamma,\eps)\preccurlyeq \frac 1 {\eps^{\max\{s,k\}}}.
   	\end{equation}
\end{prop}

\begin{proof}
    By $(i)$ in Proposition~\ref{prop:srs}, follows that we only have to prove the upper bound for the complexities relative to the cost $\costdue$.
    We will start by proving \eqref{eq:compltimek} for $\ctimecostdue$.
    In particular, by \eqref{eq:complaux} it will suffices to prove $\ctimeauxcostdue(\gamma,\delta)\preccurlyeq \delta^{\frac 1 k}$

    Let $\{z^t\}_{t\in[0,T]}$ be a continuous coordinate family for $\gamma$ adapted to $\drift$.
    Let $\tilde \gamma_t(\xi)=e^{-(\xi-t)\drift}(\gamma(\xi))$. 
    Then,  since $z^t_*\drift=\partial_{z_\ell}$, it holds
    \begin{equation}
    	\label{eq:changeofcoord}
    	z_\ell^t(\tilde \gamma_t(\xi))=z_\ell^t(\gamma(\xi))-(\xi-t),\qquad z_i^t(\tilde \gamma_t(\xi))=z_i^t(\gamma(\xi)) \quad \text{for any }i\neq \ell.
    \end{equation}

    Fix $\xi>0$ sufficiently small for Proposition~\ref{prop:normcoordpath} to hold and choose a partition $0<t_1<\ldots<t_N=T$ such that $\delta/2\le t_i-t_{i-1} \le \delta$. % and that there exists $t_{i}=t_0$.
    In particular, $N\le \lceil 2 T/\delta\rceil$.
    We then select a control $u\in\lcont$ such that its trajectory $q_u$ in \eqref{cs:d}, with $q_u(0)=x$, satisfies $q_u(t_i)=\gamma(t_i)$ for any $i=1,\ldots,N$ as follows.
    For each $i$, we choose $u_i\in\xLone([t_{i-1},t_i],\real^m)$ steering  system \eqref{cs:td} from $\gamma(t_{i-1})=\tilde\gamma_{t_{i-1}}(t_{i-1})$ to $\tilde \gamma_{t_{i-1}}(t_i)$.
    Then, by \eqref{eq:split} and the definition of $\tilde\gamma_{t_{i-1}}$, the control $u_i$ steers system \eqref{cs:d} from $\gamma(t_{i-1})$ to $\gamma(t_i)$.
    
    Since by \cite[Theorem~8]{prandi2013} it holds $\val^\costdue_{\text{TD}}\le\dsrs$, by \eqref{eq:changeofcoord}, Proposition~\ref{prop:normcoordpath} and Theorem~\ref{thm:srballbox}, if $\delta$ is sufficiently small we can choose $u_i$ such that there exists $C>0$ for which
    \begin{equation}
    	\label{eq:costduedelta}
    	\begin{split}
        		\costdue(u_i,t_i-t_{i-1})& \le C\sum_{j=1}^n |z_j^{t_{i-1}}(\tilde\gamma_{t_{i-1}}(t_i))|^{\frac 1 {w_j}} \le C\sum_{j=1}^n |z_j^{t_{i-1}}(\gamma(t_i))|^{\frac 1 {w_j}}+\delta^{\frac 1 s} \\
        	    	&\le  C \left(\sum_{w_j\le k} \delta^{\frac 1 {w_j}} +\delta^{\frac 1 s} +\sum_{\substack{w_j>k}} \delta^{\frac {1} {k}} \right) \le C\delta^{\frac 1 {\max\{k,s\}}}.
        \end{split}
    \end{equation}
    Hence, we obtain that
	\begin{equation}
		\label{eq:costdelta}
    	\costdue(u,T)\le N\, \costdue(u_i,t_i-t_{i-1})\le 3C \frac T \delta \delta^{\frac{1}{\max\{k,s\}}}.
	\end{equation}
    Since the control $u$ is admissible for $\ctimeauxcostdue(\gamma,\delta)$, this implies that $\ctimeauxcostdue(\gamma,\delta)\preccurlyeq \delta^{\frac 1 {\max\{k,s\}}}$.
    This proves the first part of the theorem.

    To complete the proof for $\cneig(\gamma,\eps)$, let $\delta = \eps^{\max\{k,s\}}$. 
    Then, by Theorems~\ref{thm:srballbox} and \ref{thm:bbox}, there exists a constant $C>0$ such that $\mathcal R^{\drift}_\delta(\gamma(t),\eps)\subset \bsr(\gamma(t),C\eps)$ for any $t\in[0,T]$.
    In particular, $\dsrs(\gamma(t_i),q_u(t))\le C\eps$ for any $t\in[t_{i-1},t_i]$.
    Moreover, again by Theorem~\ref{thm:srballbox}, Proposition~\ref{prop:normcoordpath}, and the fact that $\dot\gamma(\cdot)\in\Delta^k(\gamma(\cdot))$, this choice of $\delta$ implies also that $\dsrs(\gamma(t_{i-1}),\gamma(t))\le C\eps$ for any $t\in[t_{i-1},t_i]$.
    Hence, for any $t\in[t_{i-1},t_i]$, we get
    \[
    	\dsrs(\gamma(t),q_u(t))\le \dsrs(\gamma(t_{i-1}),q_u(t))+\dsrs(\gamma(t_{i-1}),\gamma(t))\le 2C\eps.
    \]
    Thus, $u$ is admissible for $\cneigcostdue(\gamma,C\eps)$.
    Finally, from \eqref{eq:costdelta} we get that $\cneigcostdue(\gamma,C\eps)\le \eps^{-1}\costdue(u,T)\le 3C T\eps^{-\max\{k,s\}}$, proving that $\cneig(\gamma,\eps)\preccurlyeq \eps^{-\max\{k,s\}}$.
    This completes the proof.
    \end{proof}

Now, we prove the $\succcurlyeq$ part of the statement, in the case where $\dot\gamma$ is always contained in the same stratum $\distr^k\setminus\distr^{k-1}$. 

\begin{prop}\label{prop:pathlowerbound}
	Assume that there exists $s\ge2$ such that $\drift\subset\distr^s\setminus\distr^{s-1}$.
	Let $\gamma:[0,T]\to M$ be a path, such that $\dot\gamma(t) \in \distr^k(\gamma(t))\setminus\distr^{k-1}(\gamma(t))$ for any $t\in[0,T]$.
	Moreover, if $s=k$, assume that $\drift(\gamma(t))\neq \dot\gamma(t) \mod\distr^{s-1}$ for any $t\in [0,T]$.
   	Then, it holds
	\[
		\ctimecostdue(\gamma,\eps)\succcurlyeq \ctimecostuno(\gamma,\eps)\succcurlyeq \frac 1 {\eps^{\max\{s,k\}}},\qquad 
		\cneigcostdue(\gamma,\eps) \succcurlyeq\cneigcostuno(\gamma,\eps) \succcurlyeq \frac 1 {\eps^{\max\{s,k\}}}.
	\]
\end{prop}

\begin{proof}
	By Proposition~\ref{prop:srs}, $\ctimecostuno(\gamma,\eps)\preccurlyeq \ctimecostdue(\gamma,\eps)$ and $\cneigcostuno(\gamma,\eps)\preccurlyeq \cneigcostdue(\gamma,\eps)$.
	Hence, to complete the proof it suffices to prove the asymptotic lower bound for $\ctimecostuno(\gamma,\eps)$ and $\cneigcostuno(\gamma,\eps)$.
	In the following, to lighten the notation, we write $\ctime$ and $\cneig$ instead of $\ctimecostuno$ and $\cneigcostuno$.

	\textbf{Interpolation by time complexity.}
	By \eqref{eq:complaux}, it suffices to prove that $\ctimeaux(\gamma,\delta)\succcurlyeq\delta^{\frac{1}{\max\{k,s\}}}$
	Let  $\eta>0$ and $u\in \lcont$ be a control admissible for $\ctimeaux(\gamma,\delta)$ such that 
	\begin{equation}
		\label{eq:defu}
		\costuno(u,T)=\|u\|_{\lcont}\le \frac{ \ctimeaux(\gamma,\delta) } \delta+ \eta.
	\end{equation}
	Let $N=\lceil T/\delta \rceil$ and $0=t_0<t_1<\ldots<t_N=T$ be times such that $q_u(t_i)=\gamma(t_i)$, $i=0,\ldots,N$, and $0<t_i-t_{i-1}\le \delta$.
	Observe that, up to removing some $t_i$'s, we can always assume $\delta/2\le t_{i}-t_{i-1}\le (3/ 2 )\delta$ and $N\ge \lceil (2 T)/(3\delta) \rceil$. 
	Moreover, let $u_i=u|_{[t_{i-1},t_i]}$. 
	Proceding as in the proof of Theorem~\ref{thm:srtimecomplexity}, p.~\pageref{pf:srtimecomplexity}, we get that 
	in order to show that $\ctimeaux(\gamma,\delta)\succcurlyeq \delta^{\frac 1 {\max\{k,s\}}}$ it suffices to prove
	\begin{equation}
		\label{eq:claimdelta}
		\|u_i\|_{\xLone([t_{i-1},t_i])}\ge C \delta^{\frac 1 {\max\{s,k\}}},\qquad i=1,\ldots, N.
	\end{equation}
	We distinguish three cases.

	\begin{description}
		\item[Case 1 $k>s$]
		Let $\{z^t\}$ be the continuous coordinate family for $\gamma$ adapted to $\drift$ given by Proposition~\ref{prop:normproperty}.
		Then, since $z^t_\ell(\gamma(\cdot))=0$ and $z^t_\alpha(\gamma(\xi))=\xi-t$, by Theorem~\ref{thm:bbox} it holds
		\begin{equation}\label{eq:ciao}
			\frac \delta 2 \le (t_i-t_{i-1}) = |z^{t_{i-1}}_\alpha(\gamma(t_{i}))| 
		 			\le C \|u_i\|_{\xLone([t_{i-1},t_i])}^k.
		\end{equation}
		This proves \eqref{eq:claimdelta}.

		\item[Case 2 $k < s$]
		Also in this case, let $\{z^t\}$ be the continuous coordinate family for $\gamma$ adapted to $\drift$ given by Proposition~\ref{prop:normproperty}.
		Then, by Lemma~\ref{lem:stimatau} we get 
		\begin{equation*}
			\frac \delta 2 \le t_i-t_{i-1}\le C\|u_i\|_{\xLone([t_{i-1},t_i])}^s,
		\end{equation*}	
		which immediately proves \eqref{eq:claimdelta}.

		\item[Case 3 $k=s$]
		Let $\{z^t\}_{t\in[0,T]}$ be a continuous coordinate family for $\gamma$ adapted to $\drift$.
		By the mean value theorem there exists $\xi\in[t_{i-1},t_i]$ such that
		\begin{equation}
			\label{eq:zellmean}
			z_\ell^{t_{i-1}}(\gamma(t_i))=\int_{t_{i-1}}^{t_i} (z^t_\ell)_*\dot\gamma(t)\,dt = \big((z^{t_{i-1}}_\ell)_*\dot\gamma(\xi)\big) (t_i-t_{i-1}).
		\end{equation}	
		Consider the partition $\{E_1,E_2,E_3\}$ of $[0,T]$ given by Lemma~\ref{lem:partizione} and let  $\delta\le \delta_0$.
		Then, depending to which $E_j$ belongs $t_{i-1}$, we proceed differently.
		\begin{enumerate}
			\renewcommand{\labelenumi}{(\alph{enumi})}
			\renewcommand{\theenumi}{(\alph{enumi})}
			\item \emph{$t_{i-1}\in E_1$:}
			By Lemma~\ref{lem:stimatau} and \eqref{eq:zellmean} we get
			\begin{equation*}
			    t_i-t_{i-1} \le C\|u_i\|_{\xLone([t_{i-1},t_i])}^s + z_\ell^{t_{i-1}}(\gamma(t_i))^+ = C\|u_i\|_{\xLone([t_{i-1},t_i])}^s + \big((z^{t_{i-1}}_\ell)_*\dot\gamma(\xi)\big) (t_i-t_{i-1}).
			\end{equation*}
			Then, by \eqref{eq:prope1} of Lemma~\ref{lem:partizione}, we get
			\[
				\|u_i\|_{\xLone([t_{i-1},t_i])} \ge \left(\frac{1-(z^{t_{i-1}}_\ell)_*\dot\gamma(\xi)}C \right)^{\frac 1 s} (t_i-t_{i-1})^{\frac 1 s}\ge \left(\frac \rho C\right)^{\frac 1 s} \delta^{\frac 1 s}.
			\]
			This proves \eqref{eq:claimdelta}.

			\item \emph{$t_{i-1}\in E_2$:}
			By \eqref{eq:prope2} of Lemma~\ref{lem:partizione}, \eqref{eq:zellmean} and Theorem~\ref{thm:bbox} we get
			\begin{equation*}
				m(t_i-t_{i-1}) \le |z^{t_{i-1}}_\alpha(\gamma(t_{i}))| 
		 			\le C \left(\|u_i\|_{\xLone([t_{i-1},t_i])}^s + \|u_i\|_{\xLone([t_{i-1},t_i])} |z_\ell^{t_{i-1}}(\gamma(t_i))| \right).
			\end{equation*}
			Reasoning as in \eqref{eq:costduedelta} yields that we can assume $\|u_i\|_{\xLone([t_{i-1},t_i])}\le C\delta^{\frac 1 {s}}$.
			Then, by \eqref{eq:zellmean} and letting $\delta\le (m/(2+4\rho))^s$, we get
			\begin{equation*}
				\|u_i\|_{\xLone([t_{i-1},t_i])} \ge (m-\delta^{\frac 1 s}(1+2\rho))^{\frac 1 s}(t_i-t_{i-1})^{\frac 1 s}\ge \left(\frac m 2\right)^{\frac 1 s} \delta^{\frac 1 s},		
			\end{equation*}
			proving \eqref{eq:claimdelta}.

			\item \emph{$t_{i-1}\in E_3$:}
			By Theorem~\ref{thm:bbox} it follows that 
			\begin{equation}
				\label{eq:zelldelta}
				|z_\ell^{t_{i-1}}(\gamma(t_i))|\le C\|u_i\|_{\xLone([t_{i-1},t_i])}^s+(t_i-t_{i-1}).
			\end{equation}
			Then, by \eqref{eq:zellmean} and \eqref{eq:zelldelta} we obtain
			\[
				\|u_i\|_{\xLone([t_{i-1},t_i])}\ge \left(\frac{(z^{t_{i-1}}_\ell)_*\dot\gamma(\xi)-1}C \right)^{\frac 1 s} (t_i-t_{i-1})^{\frac 1 s}\ge \left(\frac \rho C\right)^{\frac 1 s} \delta^{\frac 1 s}.
			\]
			The last inequality follows from \eqref{eq:prope3} of Lemma~\ref{lem:partizione}.
			This proves \eqref{eq:claimdelta}.
		\end{enumerate}
	\end{description}

	\textbf{Neighboring approximation complexity.}
	Fix $\eta>0$ and let $u\in\lcont$ be admissible for $\cneig(\gamma,\eps)$ and such that $\|u\|_{\lcont}\le \cneig(\gamma,\eps)+\eta$.
	Let $q_u:[0,T]\to M$ be the trajectory of $u$ with $q_u(0)=\gamma(0)$.
	Let then $N=\lceil \cneig(\gamma,\eps)+\eta\rceil$ and $0=t_0<t_1<\ldots<t_N=T$ be such that $\|u\|_{\xLone([t_{i-1},t_i])}\le \eps$ for any $i=1,\ldots, N$.
	By Proposition~\ref{prop:costole} and the fact that $q_u(t)\in\bsr(\gamma(t),\eps)$ for any $t\in[0,T]$, we can build a new control, still denoted by $u$, such that $q_u(t_i)=\gamma(t_i)$, $i=1,\ldots, N$, and $\|u\|_{\xLone([t_{i-1},t_i])}\le 3\eps$.

	Fixed a $\delta_0>0$, w.l.o.g.\ we can assume that $t_i-t_{i-1}\le\delta_0$.
	In fact, we can split each interval $[t_{i-1},t_i]$ not satisfying this property as $t_{i-1}=\xi_1<\ldots<\xi_M=t_i$, with $\xi_\nu-\xi_{\nu-1}\le\delta_0$.
	Then, as above, it is possible to modify the control $u$ so that $q_u(\xi_\nu)=\gamma(\xi_\nu)$ for any $\nu=1,\ldots,M$.
	Since $M\le\lceil T/\delta_0\rceil$ and $q_u(\cdot)\in\bsr(\gamma(\cdot),\eps)$, we have $\|u\|_{\xLone([\xi_i,\xi_{i-1}])}\le 5\eps$ and the new total number of intervals is $\le (1+\lceil T/\delta_0\rceil) \lceil\cneig(\gamma,\eps)+\eta\rceil\le C (\cneig(\gamma,\eps)+\eta)$.

	We claim that to prove $\cneig(\gamma,\eps)\succcurlyeq \eps^{-\max\{s,k\}}$, it suffices to show that there exists a constant $C>0$, independent of $u$, such that
	\begin{equation}
		\label{eq:cneigsuff}
		t_i-t_{i-1}\le C\eps^{\max\{s,k\}}, \qquad\text{for any } i=1,\ldots,N.
	\end{equation}
	In fact, since $N\le C (\cneig(\gamma,\eps)+\eta)$, this will imply that 
	\[
		T=\sum_{i=1}^N t_i-t_{i-1}\le C(\cneig(\gamma,\eps)+\eta)\eps^{\max\{s,k\}}.
	\]
	Letting $\eta\downarrow 0$, we get that $\cneig(\gamma,\eps)\succcurlyeq \eps^{-\max\{s,k\}}$, proving the claim.

	We now let $\delta_0$ sufficiently small in order to apply Lemma~\ref{lem:partizione}, Theorem~\ref{thm:bbox}, and Lemma~\ref{lem:stimatau}.
	As before, we distinguish three cases.

	\begin{description}
		\item[Case 1 $k>s$]
		Let $\{z^t\}$ be the continuous coordinate family for $\gamma$ adapted to $\drift$ given by Proposition~\ref{prop:normproperty}.
		By Theorem~\ref{thm:bbox}, using the fact that $\gamma(t_i)=q_u(t_i)$ for $i=1,\ldots,N$, we get
		\begin{equation}
			\label{eq:timeneig}
			\begin{split}
				(t_i-t_{i-1}) = |z_\alpha^{t_{i-1}}(\gamma(t_i))| \le C  \eps^k.
			\end{split}	
		\end{equation}
		This proves \eqref{eq:cneigsuff}.

		\item[Case 2 $k < s$]
		Again, let $\{z^t\}$ be the continuous coordinate family for $\gamma$ adapted to $\drift$ given by Proposition~\ref{prop:normproperty}.
		As for the interpolation by time complexity,  by Lemma~\ref{lem:stimatau} and the fact that $q_u(t_i)=\gamma(t_i)$, we get 
		\[
			(t_i-t_{i-1})  \le C \eps^s,
		\]
		thus proving \eqref{eq:cneigsuff}.

		\item[Case 3 $k=s$]
		Let $\{z^t\}_{t\in[0,T]}$ to be a continuous coordinate family for $\gamma$ adapted to $\drift$. 
		Consider the partition $\{E_1,E_2,E_3\}$ of $[0,T]$ given by Lemma~\ref{lem:partizione} and recall \eqref{eq:zellmean}.
		We distinguish three cases.
		\begin{enumerate}
			\renewcommand{\labelenumi}{(\alph{enumi})}
			\renewcommand{\theenumi}{(\alph{enumi})}
			\item \emph{$t_{i-1}\in E_1$:}
			By Lemma~\ref{lem:stimatau} and \eqref{eq:zellmean} we get
			\begin{equation*}
			    t_i-t_{i-1} 
			    	\le C\eps^s + z_\ell^{t_{i-1}}(\gamma(t_i)) = 2C\eps^s + \big((z^{t_{i-1}}_\ell)_*\dot\gamma(\xi)\big) (t_i-t_{i-1}).
			\end{equation*}
			By \eqref{eq:prope1} of Lemma~\ref{lem:partizione}, this implies
			\[
				t_i-t_{i-1} \le \left(\frac{2C}{1-(z^{t_{i-1}}_\ell)_*\dot\gamma(\xi)} \right)\eps^s \le \frac {2C}\rho \eps^s.
			\]
			Hence, \eqref{eq:cneigsuff} is proved.

			\item \emph{$t_{i-1}\in E_2$:}
			By \eqref{eq:prope2} of Lemma~\ref{lem:partizione}, \eqref{eq:zellmean} and Theorem~\ref{thm:bbox} we get
			\begin{equation*}
					m(t_i-t_{i-1}) \le |z^{t_{i-1}}_\alpha(\gamma(t_{i}))| 
						\le C \left(\eps^s + \eps |z_\ell^{t_{i-1}}(\gamma(t_i))| \right)\le C \left(\eps^s + \eps^{s+1} +\eps (t_i-t_{i-1}) \right).
			\end{equation*}
			This, by taking $\eps$ sufficiently small and enlarging $C$, implies \eqref{eq:cneigsuff}.

			\item \emph{$t_{i-1}\in E_3$:}
			By Theorem~\ref{thm:bbox} it follows that 
			\begin{equation}
				\label{eq:zelldeltaneig}
				|z_\ell^{t_{i-1}}(\gamma(t_i))|\le C\eps^s+(t_i-t_{i-1}).
			\end{equation}
			Then, by \eqref{eq:zellmean} and \eqref{eq:zelldeltaneig} we obtain
			\[
				t_i-t_{i-1}\le \frac C {(z^{t_{i-1}}_\ell)_*\dot\gamma(\xi)-1}\eps^s \le \frac C \rho \eps^s.
			\]
			The last inequality follows from \eqref{eq:prope3} of Lemma~\ref{lem:partizione}, and proves \eqref{eq:cneigsuff}.
		\end{enumerate}

	\end{description}
\end{proof}

As for the case of curves, in order to extend Proposition~\ref{prop:pathlowerbound} to paths not always tangent to the same strata, we will need the following sub-additivity property.
Let us remark that due to the definition of the path complexities, we do not need to make any assumption regarding cusps.

\begin{prop}
	\label{prop:pathsemicont}
    Let $\gamma:[0,T]\to M$ be a path and let $t_1,\,t_2\subset [0,T]$.  
    \begin{enumerate}[i.]
    \item If there exists $k\in\nat$ such that $\ctimecostuno(\gamma|_{[t_1,t_2]},\eps)\succcurlyeq \eps^{-k}$, then $\ctimecostuno(\gamma,\eps)\succcurlyeq \eps^{-k}$,
    \item $\cneigcostuno(\gamma|_{[t_1,t_2]},\eps) \preccurlyeq \cneigcostuno(\gamma,\eps)$.
    \end{enumerate}
    % \[	
    %     \ctimecostuno(\gamma|_{[t_1,t_2]},\eps) \preccurlyeq \ctimecostuno(\gamma,\eps)
    % 	\ctimeauxcostuno(\gamma|_{[t_1,t_2]},\delta) \preccurlyeq \ctimeauxcostuno(\gamma,\delta),\qquad \cneigcostuno(\gamma|_{[t_1,t_2]},\eps) \preccurlyeq \cneigcostuno(\gamma,\eps).
    % \]
\end{prop}

\begin{proof}
	\emph{Time interpolation complexity.}
    By (\ref{eq:complaux}), it suffices to prove that 	$\ctimeauxcostuno(\gamma|_{[t_1,t_2]},\delta) \preccurlyeq \ctimeauxcostuno(\gamma,\delta)$.
	Let $u\in\lcont$ be a control admissible for $\ccostcostuno(\Gamma,\eps)$, and let $0=\xi_1<\ldots<\xi_N=T$ be the times where $q_u(\xi_i)=\gamma(\xi_i)$.
	Let ${i_1}\neq i_2$ such that $t_1\le \xi_i \le t_2$ for any $i\in \{i_1,\ldots,i_2\}$.
	Observe that, by Theorems~\ref{thm:srballbox} and \ref{thm:cont}, we have $\val^\costuno(\gamma(t_1),\gamma(\xi_{i_1}))\le \dsrs(\gamma(t_1),\gamma(\xi_{i_1}))\le C\delta^{\frac 1 r}$ and $\val^\costuno(\gamma(\xi_{i_2}),\gamma(t_2))\le \dsrs(\gamma(\xi_{i_2}),\gamma(t_2))\le C\delta^{\frac 1 r}$, where $\delta$ is sufficiently small, $C$ is independent of $\delta$, and $r$ is the nonholonomic degree of the distribution.
	Thus, assuming w.l.o.g. $C\ge1$,
	\begin{equation*}
		\ctimeauxcostuno(\gamma|_{[t_1,t_2]},\delta)\le \delta {\costuno(u|_{[t_{i_1},t_{i_2}]})}  +2C \delta^{1+\frac 1 r} \le C \delta {\costuno(u)} +C \delta^{1+\frac 1 r}.
	\end{equation*}
	Taking the infimum over all controls $u$ admissible for $\ctimeauxcostuno(\gamma,\delta)$, and recalling that, by Proposition~\ref{prop:pathupperbound}, it holds $\ctimeauxcostuno(\gamma,\delta)\preccurlyeq \delta^{\frac 1 r}$, completes the proof.

	\emph{Neighboring approximation complexity.} 
	In this case, the proof is identical to the one of Proposition~\ref{prop:semicont} for the tubular approximation complexity.
	The sole difference is that here, by definition of $\cneigcostuno$, we do not need to assume the absence of cusps.
\end{proof}

We can now complete the proof of Theorem~\ref{thm:driftcomplexity}, by proving Theorem~\ref{thm:pathgeneral}.

\begin{proof}
    The proof is analogous to the one of Theorem~\ref{thm:curvegeneral}, using Propositions~\ref{prop:pathupperbound}, \ref{prop:pathlowerbound} and \ref{prop:pathsemicont}.
\end{proof}

\bibliographystyle{amsplain}
\bibliography{complexity}

\end{document}